\documentclass[11pt]{amsart}

\usepackage{amsmath,amssymb,amsthm}
\usepackage[pdftex]{color,graphicx}
\usepackage[utf8]{inputenc} 
\usepackage{graphics}
\usepackage{enumerate}
\usepackage{array}
\usepackage[english]{babel}
\usepackage{verbatim}
\usepackage{hyperref}
\usepackage{caption}
\usepackage{tikz-cd}

\usepackage{pdflscape}
\usepackage{rotating}
\usepackage{graphicx}

\usepackage{amsmath}
\usepackage{amsfonts}
\usepackage{amssymb}
\usepackage{amsthm}
\usepackage{mathrsfs}
\usepackage{mathtext}
\usepackage{mathtools}

\usepackage[all,knot]{xy}

\newcommand{\K}{\mathbb{K}}

\newcommand{\ot}{\otimes}

\DeclareMathOperator{\HH}{HH}
\DeclareMathOperator{\HHom}{Hom}

\numberwithin{equation}{section}

\newtheorem{defi}[equation]{Definition}
\newtheorem{rema}[equation]{Remark}
\newtheorem{theo}[equation]{Theorem}
\newtheorem*{theo*}{Theorem}
\newtheorem{prop}[equation]{Proposition}
\newtheorem{coro}[equation]{Corollary}

\newtheorem{lemma}[equation]{Lemma}
\newtheorem{exam}[equation]{Example}

\newtheorem{quest}[equation]{Question}

\title[Bracket structure on Hochschild cohomology of Koszul quiver algebras]{Bracket structure on Hochschild cohomology of Koszul quiver algebras using homotopy liftings}

\author[T.\ Oke]{Tolulope Oke}
\address{Department of Mathematics, Texas A\&M University, 
College Station, Texas 77843, USA}
\email{tolu\_oke@tamu.edu}

\date{\today}

\begin{document}

\maketitle
\thispagestyle{empty}

\begin{abstract}
We present the Gerstenhaber algebra structure on the Hochschild cohomology of Koszul algebras defined by quivers and relations using the idea of homotopy liftings. E.L. Green, G. Hartman, E.N. Marcos and \O. Solberg provided a canonical way of constructing a minimal projective bimodule resolution of a Koszul quiver algebra. The resolution has a comultiplicative structure which we use to define homotopy lifting maps. We first present a short example that demonstrates the theory. We then study the Gerstenhaber algebra structure on Hochschild cohomology of a family of bound quiver algebras, some members of which are counterexamples to the Snashall-Solberg finite generation conjecture. We give examples of homotopy lifting maps for degree $2$ and degree $1$ cocycles and draw connections to derivation operators. As an application, we describe Hochschild 2-cocycles satisfying the Maurer-Cartan equation.

\end{abstract}

\let\thefootnote\relax
\let\thefootnote\relax\footnote{Partially supported by NSF grant 1665286 and 2001163 during the summer semester of 2020 and Spring 2021.}
 

\section{Introduction}\label{introduction}
\qquad The Hochschild cohomology $\HH^*(\Lambda)$ of an associative algebra $\Lambda$ possesses a multiplicative map making it a graded commutative ring. This map is known as the cup product. M. Gerstenhaber introduced a bracket on $\HH^*(\Lambda)$ giving it a second multiplicative structure. This bracket, now known as the Gerstenhaber bracket, together with the cup product makes $\HH^*(\Lambda)$ into a Gerstenhaber algebra. The bracket plays an important role in the theory of deformation of algebras.

The bracket structure introduced by M. Gerstenhaber was defined using the bar resolution. This definition is useful theoretically but not easily accessible for computational purposes. For instance to compute the bracket on $\HH^*(\Lambda)$ using an arbitrary resolution, appropriate chain maps between the resolution and the bar resolution would have to be constructed. Morphisms defined on the resolution would then have to be carried over to the bar resolution and vice versa. This process is not always easy. 

Several works have been carried out in interpreting the original definition of the bracket given by M. Gerstenhaber using arbitrary resolutions. For example, in 2004, B. Keller realized Hochschild cohomology as the Lie algebra of the derived Picard group. M. Su\'{a}rez-\'{A}lvarez showed in \cite{MSA} that the bracket of a degree 1 cocycle and any cocycle can be realized from derivation operators associated to the degree 1 cocycles expressed on arbitrary projective resolutions. In~\cite{NW}, C. Negron and S. Witherspoon introduced the idea  of a contracting homotopy for resolutions that are differential graded coalgebras. In~\cite{VOL}, Y. Volkov generalized the method introduced by C. Negron and S. Witherspoon to arbitrary resolutions. He introduced the idea of homotopy lifting by defining the bracket in terms of homotopy lifting maps on any projective bimodule resolution.

In this work, we present the Gerstenhaber algebra structure on Hochschild cohomology of Koszul algebras defined by quivers and relations using homotopy lifting technique introduced by Y. Volkov in \cite{VOL}. We construct explicit examples of  homotopy lifting maps for some cocycles. Our cocycles and homotopy lifting maps are defined using the resolution $\K$ introduced in~\cite{ROKA} by E. L. Green, G. Hartman, E. N. Marcos and \O. Solberg. We discuss the construction of $\K$ in details in the Preliminaries section. There are two necessary and sufficient conditions that make a map a homotopy lifting map. For the resolution $\K$, the second condition is satisfied because $\K$ is a differential graded coalgebra and more generally because the algebra $\Lambda$ is Koszul. 

One of the examples we use to illustrate the theory is a family of quiver algebras. Some members of this family are counterexamples to the Snashall-Solberg finite generation conjecture. The Snashall-Solberg finite generation conjecture asserts that for finite dimensional algebras, Hochschild cohomology modulo nilpotents is finitely generated. R. Hermann asked whether or not Hochschild cohomology modulo the weak Gerstenhaber ideal generated by nilpotent elements is finitely generated. The theory provided in this work gives the necessary tools neeeded to address these questions.

Furthermore, there is an algorithmic approach to computing the resolution $\K$, for instance  in \cite{AAR}. This means that perhaps, it might be possible to employ our method to develop algorithms that can compute the bracket on Hochschild cohomology for Koszul algebras defined by quivers and relations.

Through out this article, $k$ is taken to be a field and $\otimes = \otimes_k$ unless otherwise specified. $\Lambda=kQ/I$ is taken to be a graded Koszul quiver algebra, where $Q$ is a finite quiver and $I$ is an admissible ideal of the quiver algebra $kQ$. For $R=kQ$, and suppose $\Lambda = \oplus_{i\geq 0}\Lambda_i$ is a grading of $\Lambda$, it was shown in \cite{ROKA} that there are integers $t_n$  and uniform elements $\{f^n_i\}_{i=0}^{t_n}\in R$ such that the minimal projective resolution $\mathbb{L}\rightarrow\Lambda_0$ can be given in terms of filtration of right ideals involving the $f^n_i$. It was shown that the elements $\{f^n_i\}_{i=0}^{t_n}$ satisfy a ``comultiplicative structure", i.e. there are scalars $c_{pq}(n,i,r)$ such that $f^n_i = \sum_{p=0}^{t_r}\sum_{q=0}^{t_{n-r}}c_{pq}(n,i,r) f^r_p f^{n-r}_q.$ From this comultiplicative equation, one constructs $\K_{\bullet}\rightarrow\Lambda,$ a minimal projective $\Lambda^e$-resolution of $\Lambda$, where for each $n$, the $\Lambda^e$-module $\K_n$ has free basis elements $\{\varepsilon^n_i\}_{i=0}^{t_n}$~\cite{ROKA}. The main results of this work are the following:
\begin{enumerate}
\item[(i).] Theorems~\ref{homotopylifting00}, \ref{homotopylifting1}, and \ref{homotopylifting2} which presents definitions of homotopy lifting maps for cocycles taking free basis elements $\varepsilon^m_r$ to an idempotent, a path of length 1 and a path of length 2 respectively.
\item[(ii).] Theorem \ref{Gerstenbrack} providing a combinatorial description of the Gerstenhaber bracket of any two Hochschild cochains and Theorem \ref{Gerstenbrack1} which gives a general Gerstenhaber algebra structure on Hochschild cohomology of Koszul quiver algebras.
\item[(iii).] Theorem \ref{Maurer-cartan-theorem} specifying Hochschild 2-cocycles that satisfy the Maurer-Cartan equation.
\item[(iv).] Proposition \ref{homotopylift_doperators} showing examples of homotopy lifting maps for degree $1$-cocycles that are also derivation operators when regarded as $k$-linear maps. We raised Question \ref{question_homo-doperator} which asks whether the observation in Proposition \ref{homotopylift_doperators} is generally true.
\end{enumerate}
The results of item (i) above assume certain conditions on the scalars $c_{pq}(n,i,r)$ appearing in the multiplicative equation as well as some scalars $b_{m,r}(m-n+1,s)$ coming from homotopy lifting maps. We present these conditions in Equation~\eqref{main-result}. This article is organized in the following way:

\tableofcontents

\section{Preliminaries}\label{prelim}
The Hochschild cohomology of an associative $k$-algebra $\Lambda$ was originally defined using the following projective resolution known as the bar resolution.
\begin{equation}\label{bar}
\mathbb{B}_{\bullet}:  \qquad\cdots \rightarrow \Lambda^{\ot (n+2)}\xrightarrow{\;\delta_n\; }\Lambda^{\ot (n+1)}\xrightarrow{\delta_{n-1}} \cdots \xrightarrow{\;\delta_2\;}\Lambda^{\ot 3}\xrightarrow{\;\delta_1\;}\Lambda^{\ot 2}\;(\;\xrightarrow{\mu} \Lambda)
\end{equation}
where $\mu$ is multiplication and the differentials $\delta_n$ are given by
\begin{equation}\label{bar-diff}
\delta_n(a_0\otimes a_1\otimes\cdots\otimes a_{n+1}) = \sum_{i=0}^n (-1)^i a_0\otimes\cdots\otimes a_ia_{i+1}\otimes\cdots\otimes a_{n+1}
\end{equation}
for all $a_0, a_1,\ldots, a_{n+1}\in \Lambda$. This resolution consists of $\Lambda$-bimodules or left modules over the enveloping algebra $\Lambda^e = \Lambda\ot\Lambda^{op}$, where $\Lambda^{op}$ is the opposite algebra. The resolution is sometimes written $\mathbb{B}_{\bullet}\xrightarrow{\mu}\Lambda$ with $\mu$ referred to as the augmentation map. Let $M$ be a finitely generated left $\Lambda^e$-module, the space of Hochschild $n$-cochains with coefficients in $M$ is obtained by applying the functor $\HHom_{\Lambda^e}(-,M)$ to the complex $\mathbb{B}_{\bullet}$, and then taking the cohomology of the resulting cochain complex. We define 
$$ \HH^*(\Lambda,M) := \bigoplus_{n\geq 0} \HH^n(\Lambda,M)=\bigoplus_{n\geq 0} \text{H}^n(\HHom_{\Lambda^e}(\mathbb{B}_{n},M)) $$
to be the Hochschild cohomology ring of $\Lambda$ with coefficients in $M$. Letting $M=\Lambda$, we define $\HH^*(\Lambda) :=\HH^*(\Lambda,\Lambda) $ to be the Hochschild cohomology of $\Lambda$. An element $\chi\in\HHom_{\Lambda^e}(\mathbb{B}_m,\Lambda)$ is a cocycle if $(\delta^{*}(\chi))(\cdot) := \chi\delta(\cdot)=  0.$ There is an isomorphism of the abelian groups $\HHom_{\Lambda^e}(\mathbb{B}_m,\Lambda)\cong\HHom_{k}(\Lambda^{\otimes m},\Lambda),$ so we can also view $\chi$ as an element of $\HHom_{k}(\Lambda^{\otimes m},\Lambda)$.

The Gerstenhaber bracket of two cocycles $\chi\in\HHom_{k}(\Lambda^{\otimes m},\Lambda)$ and $\theta\in \HHom_{k}(\Lambda^{\otimes n},\Lambda)$ at the chain level is given by
\begin{equation}\label{gbrac}
[\chi,\theta] = \chi\circ \theta - (-1)^{(m-1)(n-1)} \theta\circ \chi
\end{equation}
where $\chi\circ \theta = \sum_{j=1}^m (-1)^{(n-1)(j-1)}\chi\circ_j \theta$ with
\begin{multline*}
(\chi\circ_j \theta)(a_1\otimes \cdots \otimes a_{m+n-1}) \\
= \chi(a_1\otimes\cdots\otimes a_{j-1}\otimes \theta(a_j\otimes\cdots\otimes a_{j+n-1}) \otimes a_{j+n}\otimes \cdots\otimes a_{m+n-1}). \end{multline*}
This induces a well defined map  $[\cdot\;,\cdot] : \HH^{m}(\Lambda) \times \HH^{n}(\Lambda) \rightarrow \HH^{m+n-1}(\Lambda)$ on cohomology.

\textbf{Quiver algebras:} A quiver is a directed graph with the allowance of loops and multiple arrows. A quiver $Q$ is sometimes denoted as a quadruple $(Q_0,Q_1,o,t)$ where $Q_0$ is the set of vertices in $Q$, $Q_1$ is the set of arrows in $Q$, and $o,t: Q_1 \longrightarrow Q_0$ are maps which assign to each arrow $a\in Q_1$, its origin vertex $o(a)$ and terminal vertex $t(a)$ in $Q_0$. A path in $Q$ is a sequence of arrows $a = a_1 a_2 \cdots a_{n-1} a_n $ such that the terminal vertex of $a_{i}$ is the same as the origin vertex of $a_{i+1}$, using the convention of concatenating paths from left to right. The quiver algebra or path algebra $kQ$ is defined as a vector space having all paths in $Q$ as a basis. Vertices are regarded as paths of length $0$, an arrow is a path of length $1$, and so on. We take multiplication on $kQ$ as concatenation of paths. Two paths $a$ and $b$ satisfy $ab=0$ if $t(a)\neq o(b)$. This multiplication defines an associative algebra over $k$. 
By taking $kQ_i$ to be the $k$-vector subspace of $kQ$ with paths of length $i$ as basis, $kQ = \bigoplus_{i\geq 0} kQ_i$ can be viewed as an $\mathbb{N}$-graded vector space. A relation on a quiver $Q$ is a linear combination of paths from $Q$ having the same origin vertex and terminal vertex. A quiver together with a set of relations is called a quiver with relations. Letting $I$ be an ideal of the path algebra $kQ$, we denote by $(Q,I)$ the quiver $Q$ with relations $I$. The quotient $\Lambda  = kQ/I$ is called the quiver algebra associated with $(Q,I)$. Let $\Lambda = \bigoplus_{i\geq 0}\Lambda_i$ be a grading on $\Lambda$. If $\Lambda$ is Koszul and $\Lambda_0$ is isomorphic to $k$ or copies of $k$, $\Lambda_0$ has a linear graded projective resolution $\mathbb{L}$ as a right $\Lambda$-module \cite{HCA,MV}.

An algorithmic approach to finding such a minimal projective resolution $\mathbb{L}$ of $\Lambda_0$ was given in~\cite{AAR}. The modules $\mathbb{L}_{n}$ are right $\Lambda$-modules for each $n$. There is a ``comultiplicative structure" on $\mathbb{L}$ and this structure was used to find a minimal projective resolution $\K\rightarrow\Lambda$ of modules over the enveloping algebra of $\Lambda$ in~\cite{ROKA}. A non-zero element $x\in kQ$ is called uniform if it is a linear combination of paths each having the same origin vertex and the same terminal vertex: In other words, $x=\sum_{j} c_j w_j$ with scalars $c_j\neq 0$ for all $j$ and each path $w_j$ are of equal length having the same origin vertex and the same terminal vertex. For $R=kQ$, it was shown in \cite{AAR} that there are integers $t_n$  and uniform elements $\{f^n_i\}_{i=0}^{t_n}$ such that the right projective resolution $\mathbb{L}\rightarrow\Lambda_0$ is obtained from a filtration of $R$. This filtration is given by the following nested family of right ideals: 
$$\cdots \subseteq \bigoplus_{i=0}^{t_n} f^n_iR \subseteq \bigoplus_{i=0}^{t_{n-1}} f^{n-1}_iR\subseteq \cdots \subseteq\bigoplus_{i=0}^{t_1} f^1_iR\subseteq \bigoplus_{i=0}^{t_0} f^0_iR = R$$
where for each $n$, $\mathbb{L}_n = \bigoplus_{i=0}^{t_n} f^n_iR / \bigoplus_{i=0}^{t_n} f^n_iI$ and the differentials on $\mathbb{L}$ are induced by the inclusions $\bigoplus_{i=0}^{t_n} f^n_iR\subseteq \bigoplus_{i=0}^{t_{n-1}} f^{n-1}_iR$. Furthermore, it was shown in \cite{AAR} that with some choice of scalars, the $\{f^n_i\}_{i=0}^{t_n}$ satisfying the comultiplicative equation of \eqref{comult-struc} make $\mathbb{L}$ minimal. In other words, for $0\leq i\leq t_n$, there are scalars $c_{pq}(n,i,r)$ such that
\begin{equation}\label{comult-struc}
f^n_i = \sum_{p=0}^{t_r}\sum_{q=0}^{t_{n-r}}c_{pq}(n,i,r) f^r_p f^{n-r}_q
\end{equation}
holds and $\mathbb{L}$ is a minimal resolution. To construct the above multiplicative equation for example, we can take $\{f^0_i\}_{i=0}^{t_0}$ to be the set of vertices, $\{f^1_i\}_{i=0}^{t_1}$ to be the set of arrows, $\{f^2_i\}_{i=0}^{t_2}$ to be the set of uniform relations generating the ideal $I$, and define $\{f^n_i\}_{i=0}^{t_n}(n\geq 3)$ recursively, that is in terms of $f^{n-1}_i$ and $f^1_j$. We present the comultiplicative structure of a family of quiver algebras in Section~\eqref{wexamples}. It was first presented in \cite{TNO} where we used the ideas in this work to show that for certain algebras from the family, the Hochschild cohomology ring modulo the weak Gerstenhaber ideal generated by homogeneous nilpotent elements is not finitely generated.

The resolution $\mathbb{L}$ and the comultiplicative structure~\eqref{comult-struc} were used to construct a minimal projective resolution $\K\rightarrow \Lambda$ of modules over the enveloping algebra $\Lambda^e=\Lambda\otimes \Lambda^{op}$ on which we now define Hochschild cohomology. This minimal projective resolution $\K$ of $\Lambda^e$-modules associated to $\Lambda$ was given in \cite{ROKA} and now restated with slight notational changes below.

\begin{theo}\cite[Theorem 2.1]{ROKA}\label{mainres}
Let $\Lambda=KQ/I$ be a Koszul algebra, and let $\{f^n_i\}_{i=0}^{t_n}$ define a minimal resolution of $\Lambda_0$ as a right $\Lambda$-module. A minimal projective resolution $(\K, d)$ of $\Lambda$ over $\Lambda^e$ is given by
\begin{equation*}
\mathbb{K}_n = \bigoplus_{i=0}^{t_n}\Lambda o(f_i^n)\otimes_k t(f_i^n)\Lambda
\end{equation*} 
for $n\geq 0$, where the differential $d_n:\K_n\xrightarrow{}\K_{n-1}$ applied to $\varepsilon^n_i = (0,\ldots,0,o(f^n_i)\otimes_k t(f^n_i),0,\ldots,0),$ $0\leq i\leq t_n$  where $o(f^n_i)\otimes_k t(f^n_i)$ is in the $i$-th position is given by
\begin{equation}\label{diff-k}
d_n(\varepsilon^n_i ) = \sum_{j=0}^{t_{n-1}}\Big( \sum_{p=0}^{t_1}c_{p,j}(n,i,1) f_p^1 \varepsilon^{n-1}_j 
+ (-1)^n \sum_{q=0}^{t_1}c_{j,q}(n,i,n-1)\varepsilon^{n-1}_j  f_q^1 \Big) 
\end{equation}
and $d_0:K_0\xrightarrow{}\Lambda$ is the multiplication map. In particular, $\Lambda$ is a linear module over $\Lambda^e$.
\end{theo}


We note that for each $n$ and $i$, $\{\varepsilon^n_i\}_{i=0}^{t_n}$ is a free basis of $\K_n$ as a $\Lambda^e$-module. The scalars $c_{p,j}(n,i,r)$ are those appearing in~\eqref{comult-struc} and $f^1_{*} := \overline{f^1_{*}}$ is the residue class of $f^1_{*}$ in  $\bigoplus_{i=0}^{t_1} f^1_iR / \bigoplus_{i=0}^{t_n} f^1_i I$. Using the comultiplicative structure of Equation~\eqref{comult-struc}, it was shown in~\cite{MSKA} that the cup product on the Hochschild cohomology ring of a Koszul quiver algebra has the following description.
\begin{theo*}[See~\cite{MSKA}, Theorem 2.3]\label{hochschildcup}
Let $\Lambda = kQ/I$  be a Koszul algebra over a field $k$,  where $Q$ is a finite quiver and $I\subseteq J^2$. Suppose that $\eta:\K_n\rightarrow\Lambda$ and $\theta:\K_m\rightarrow\Lambda$ represent elements in $\HH^*(\Lambda)$ and are given by $\eta(\varepsilon^n_i) = \lambda_i$ for $i=0,1,\dotsc,t_n$ and  $\theta(\varepsilon^m_i) = \lambda_i'$ for $i=0,1,\dotsc,t_m.$ Then $\quad \eta\smile \theta : \K_{n+m}\rightarrow\Lambda$ can be expressed as
\begin{equation*}
(\eta\smile \theta)(\varepsilon^{n+m}_j) = \sum_{p=0}^{t_n}\sum_{q=0}^{t_m}c_{pq}(n+m,i,n)\lambda_p\lambda_q' ,  
\end{equation*}
for $j= 0,1,2,\dotsc, t_{n+m}.$
\end{theo*}

In Section \ref{major}, we present in Theorem \ref{Gerstenbrack1} a similar generalized Gerstenhaber algebra structure on Hochschild cohomology of Koszul algebras defined by quivers and relations under certain conditions. Our presentation uses the scalars coming from the comulplicative structure of Equation~\eqref{comult-struc} as well as scalars coming from homotopy lifting maps.

\textbf{The reduced bar resolution of $\Lambda = kQ/I$} as presented in~\cite[Section 1]{MSKA}: We recall the definition of the reduced bar resolution of algebras defined by quivers and relations. If $\Lambda_0$ is isomorphic to $m$ copies of $k$, take $\{e_1,e_2,\ldots,e_m\}$ to be a complete set of primitive orthogonal central idempotents of $\Lambda$. In this case $\Lambda$ is not necessarily an algebra over $\Lambda_0$. If $\Lambda_0$ is isomorphic to $k$, then $\Lambda$ is an algebra over $\Lambda_0$. The reduced bar resolution $(\mathcal{B}, \delta)$, where  $\mathcal{B}_n := \Lambda^{\ot_{\Lambda_0}(n+2)}$ is the $(n+2)$-fold tensor product of $\Lambda$ over $\Lambda_0$ and uses the same differential on the usual bar resolution presented in Equation~\eqref{bar-diff}. The resolution $\K$ can be embedded naturally into the reduced bar resolution $\mathcal{B}$. There is a map $\iota:\K\rightarrow\mathcal{B}$ defined by $\iota(\varepsilon^n_r) = 1\ot \widetilde{f^n_r}\ot 1$ such that $\delta\iota=\iota d$, where 
\begin{equation}\label{the-f}
\widetilde{f^n_j} = \sum c_{j_1j_2\cdots j_n} f^1_{j_1}\ot f^1_{j_2}\ot \cdots\ot f^1_{j_n}\quad\text{ if }\quad f^n_j = \sum c_{j_1j_2\cdots j_n} f^1_{j_1} f^1_{j_2} \cdots f^1_{j_n}
\end{equation}
for some scalar $c_{j_1j_2\cdots j_n}$. It was shown in \cite[Proposition 2.1]{MSKA} that $\iota$ is indeed an embedding. By taking $\Delta:\mathcal{B}\rightarrow\mathcal{B}\ot_{\Lambda}\mathcal{B}$ to be the following comultiplicative map (or diagonal map) on the bar resolution,
\begin{equation}\label{diag-bar}
\Delta(a_0\otimes\cdots\otimes a_{n+1}) = \sum_{i=0}^n (a_0\otimes\cdots\otimes a_i\ot 1)\ot_{\Lambda}( 1\ot a_{i+1}\otimes\cdots\otimes a_{n+1}).
\end{equation}
it was also shown in~\cite[Proposition 2.2]{MSKA} that the comultiplicative map $\Delta_{\K}:\K\rightarrow\K\ot_{\Lambda}\K$ on the complex $\K$ has the following form.
\begin{equation}\label{diag-k}
\Delta_{\K}(\varepsilon^n_r) = \sum_{v=0}^{n}\sum_{p=0}^{t_v}\sum_{q=0}^{t_{n-v}}c_{p,q}(n,r,v)\varepsilon^v_p\ot_{\Lambda}\varepsilon^{n-v}_q.
\end{equation}
The compatibility of $\Delta_{\K}, \Delta$ and $\iota$ means that $(\iota\ot \iota)\Delta_{\K} = \Delta \iota$ where $(\iota\ot \iota)(\K\ot_{\Lambda}\K) = \iota(\K)\ot_{\Lambda}\iota(\K)\subseteq \mathcal{B}\ot_{\Lambda}\mathcal{B}.$

\section{Comparison morphisms vs homotopy liftings}\label{comp}
\textbf{Comparisom morphisms:} Let $(\mathbb{P},\mu_P)$ and $(\mathbb{Q},\mu_Q)$ be two $A^e$-projective resolutions of a $k$-algebra $A$, where $\mu_P$ and $\mu_Q$ are the augmentation maps. Since $\mathbb{P}$ is exact at $i$ for all $i>0,$ there are chain maps $\Phi_{P}^{Q}:\mathbb{P}\xrightarrow{}\mathbb{Q}$ such that $\mu_Q\Phi_{P}^{Q} = \mu_P$. A similar argument implies that there are chain maps $\Phi_{Q}^{P}:\mathbb{Q}\xrightarrow{}\mathbb{P}$ such that $\mu_P\Phi_{Q}^{P} = \mu_Q$. So $\Phi_{P}^{Q}$ and $\Phi_{Q}^{P}$ are maps between chain complexes. Let $(\Phi_{Q}^{P})^{*}:\text{H}^{*}(\mathbb{P},A)\xrightarrow{}\text{H}^{*}(\mathbb{Q},A)$ and $(\Phi_{P}^{Q})^{*}:\text{H}^{*}(\mathbb{Q},A)\xrightarrow{}\text{H}^{*}(\mathbb{P},A)$ be the induced map on cohomology. It is easy to see that $(\Phi_{P}^{Q})^{*}(\Phi_{Q}^{P})^{*} = (\Phi_{Q}^{P}\Phi_{P}^{Q})^{*} = 1_{\text{H}^{*}(\mathbb{P},A)}$. This is because for any $F\in \text{H}^{*}(\mathbb{P},A)$, $(\Phi_{Q}^{P}\Phi_{P}^{Q})^{*} F = F\Phi_{Q}^{P}\Phi_{P}^{Q}\sim F 1_P = F$ since $\Phi_{Q}^{P}\Phi_{P}^{Q}$ is homotopic to the identity map on the complex $\mathbb{P}$. Analogously, we have $(\Phi_{Q}^{P})^{*}(\Phi_{P}^{Q})^{*} = (\Phi_{P}^{Q}\Phi_{Q}^{P})^{*} = 1_{\text{H}^{*}(\mathbb{Q},A)}.$ So Hochschild cohomology does not depend on the choice of projective $A^e$-resolution. 

Now suppose the cup product and the Gerstenhaber bracket on $(\HH^*(A),\smallsmile, [\;,\;])$ are defined using the resolution $\mathbb{Q},$ we may define the cup product $\smallsmile_{\Phi}$ with respect to the resolution $\mathbb{P}$ on two representatives $F:\mathbb{P}_m\xrightarrow{}A$ and $G:\mathbb{P}_n\xrightarrow{}A$ by $F\smallsmile_{\Phi} G = (F\Phi_{Q}^{P}\smallsmile G\Phi_{Q}^{P})\Phi_{P}^{Q}.$ Similarly, we can define the Gerstenhaber bracket by $[F,G]_{\Phi}  = [F\Phi_{Q}^{P}, G\Phi_{Q}^{P}]\Phi_{P}^{Q}.$

Let $\mathbb{Q}$ be the reduced bar resolution $\mathcal{B}$ and let $\mathbb{P}$ be the resolution $\K$. We can take $\Phi_{K}^{B} := \iota,$ the map defined just before Equation \eqref{the-f}. It is not obvious how to define $\Phi_{B}^{K}$, the chain map from the reduced bar resolution to the resolution $\K$. The method of homotopy lifting overcomes this barrier because it does not require definining the chain $\Phi_{B}^{K}$. We now recall important notions about the homotopy lifting technique given in \cite{HCA} and \cite{VOL}.


\textbf{Homotopy lifting:} 
Let $\mathbb{P}\xrightarrow{\mu_P}A$ be a projective resolution of $A$ as an $A^e$-module with differential $d^P$ and augmentation map $\mu_P.$ We take $\textbf{d}$ to be the differential on the Hom complex $\HHom_{\Lambda^e}(\mathbb{P},\mathbb{P})$ defined for any degree $n$ map $g:\mathbb{P}\rightarrow\mathbb{P}[-n]$ as 
$$\textbf{d}(g):= d^Pg - (-1)^ng d^P$$
where $\mathbb{P}[-n]$ is a shift in homological dimension with $(\mathbb{P}[-n])_m = \mathbb{P}_{m-n}$.
In the following definition, the notation $\sim$ is used for two cocycles that are cohomologous, that is, they differ by a coboundary. 
\begin{defi}\label{homolift}
Let $\Delta_{\mathbb{P}}$ be a chain map (also called the diagonal map) lifting the identity map on $A\cong A\ot_{A}A$ and suppose that $\eta\in\HHom_{A^e}(\mathbb{P}_n, A)$ is a cocycle. A module homomorphism $\psi_\eta :\mathbb{P} \rightarrow \mathbb{P}[1- n]$ is called a \textbf{homotopy lifting} map of $\eta$ with respect to $\Delta_{\mathbb{P}}$  if
\begin{align}\label{homo-defi}
\textbf{d}(\psi_\eta) &= (\eta\ot 1_{P} - 1_{P}\ot \eta)\Delta_{\mathbb{P}} \qquad\text{ and } \\ 
\mu_P\psi_\eta &\sim \; (-1)^{n-1} \eta \psi  \notag
\end{align}
for some $\psi:\mathbb{P}\rightarrow \mathbb{P}[1]$ for which $\textbf{d}(\psi) =  (\mu_P\ot 1_{P} - 1_{P}\ot \mu_P)\Delta_{\mathbb{P}}$. 
\end{defi}

\begin{exam}\label{example-homo}
Let $\mathbb{P}$ be the bar or reduced resolution $\mathcal{B}$. Suppose that $g\in\HHom_{\Lambda^e}(\mathcal{B}_{n},A)\cong\HHom_{k}(A^{\ot n},A)$. Then one way to define a homotopy lifting map $\psi_g:\mathcal{B}_{m+n-1}\longrightarrow\mathcal{B}_{m}$ for the cocycle $g$ is the following:
\begin{multline*}
\psi_g(1\ot a_1\ot\cdots\ot a_{m+n-1}\ot 1) \\
= \sum_{i=1}^{m}(-1)^{(m-1)(i-1)}1\ot a_1\ot\cdots\ot a_{i-1}\ot 
g(a_{i}\ot\cdots\ot a_{i+n-1})
\ot a_{i+n}\ot\cdots\ot a_{m+n-1}\ot 1.
\end{multline*}
\end{exam}

\begin{rema}\label{rema-koszul-homo}
The resolution $\K$ for Koszul algebras is a differential graded coalgebra i.e. $(\Delta_{\K}\ot 1_{\K})\Delta_{\K} = (1_{\K}\ot \Delta_{\K})\Delta_{\K}$ and $(d\ot 1 + 1\ot d)\Delta_{\K} = \Delta_{\K}d$. Furthermore, the augmentation map $\mu:\K\rightarrow\Lambda$, which can be thought of as a counit makes $(\mu\ot 1_{\K})\Delta_{\K} - (1_{\K}\ot \mu)\Delta_{\K} = 0.$ We can therefore take $\psi = 0$ in Equation \eqref{homo-defi}, so that we have $\mu\psi_\eta \sim 0.$ Next, we set $\psi_\eta(\K_{n-1})=0$ and the second hypothesis of Definition~\ref{homolift} is satisfied. Later in Section \ref{major}, when we define homotopy lifting maps using $\K$, it will be sufficient to only verify the first Equation in \eqref{homo-defi}. We now give a theorem of Y. Volkov which is equivalent to the definition of the bracket presented earlier in Equation \eqref{gbrac}.
\end{rema}

\begin{theo}\cite[Theorem 4]{VOL}\label{gbrac1}
Let $(\mathbb{P},\mu_P)$ be a $\Lambda^e$-projective resolution of $\Lambda$, and let $\Delta_{\mathbb{P}}:\mathbb{P}\xrightarrow{} \mathbb{P}\ot_{\Lambda}\mathbb{P}$ be a diagonal map. Let $\eta:\mathbb{P}_n\xrightarrow{}\Lambda$ and $\theta:\mathbb{P}_m\xrightarrow{}\Lambda$ represent some cocycles. Suppose that $\psi_\eta$ and $\psi_\theta$ are homotopy liftings for $\eta$ and $\theta$ respectively. Then the Gerstenhaber bracket of the classes of $\eta$ and $\theta$ can be represented by the class of the element 
\begin{equation*}
[\eta, \theta]_{\Delta_{\mathbb{P}}} = \eta\psi_\theta - (-1)^{(m-1)(n-1)}\theta\psi_\eta.
\end{equation*}
\end{theo}
\begin{proof}
See \cite[Section 6.3]{HCA} or \cite{VOL}.
\end{proof}




\section{General bracket structure}\label{major}
In this section, we define homotopy lifting maps for Koszul algebras defined by quivers and relations and prove that under certain conditions on some scalars, they are indeed homotopy lifting maps. We present examples of these general definitions of homotopy lifting maps in Sections \ref{sexample} and  \ref{wexamples}. The main proofs are presented in  Theorems~\ref{homotopylifting00},\ref{homotopylifting1} and~\ref{homotopylifting2}. We build upon the results presented in the preliminaries in Section~\ref{prelim}. In particular, we refer to the comultiplicative structure of Equation~\eqref{comult-struc}, the minimal projective resolution $\K$ given in \eqref{mainres} and the differentials $d$ on $\K$ given by Equation~\eqref{diff-k}.

\textbf{Notation:} We will use the following standard notation. Since the set $\{\varepsilon^n_i\}_{i=0}^{t_n},$ forms a basis for $\K_n$, for any module homomorphism $\theta:\K_n\rightarrow\Lambda_q$ taking $\varepsilon^n_i$ to $\lambda_i$, $i=0,1,\ldots,t_n$, we use the notation $ \theta = \begin{pmatrix} \lambda_{0} & \lambda_{1} & \cdots  &  \lambda_{t_n}\end{pmatrix}$
 to encode this information. If $\theta$ takes $\varepsilon^n_i$ to $\lambda$, and every other basis elements to $0$, we write $ \theta = \begin{pmatrix} 0 & \cdots & 0 & (\lambda)^{(i)} & 0 & \cdots  & 0\end{pmatrix}$. Our results in this section can be summarized as considering cases where $\lambda$ is an idempotent, a path of length 1, or a path of length 2. These results are presented in Theorems \ref{homotopylifting00}, \ref{homotopylifting1} and \ref{homotopylifting2}. We start with the case where $\lambda$ is an idempotent.

\begin{theo}\label{homotopylifting00}
Let $\Lambda = kQ/I$ be a Koszul quiver algebra. Suppose that $\eta :\K_n\rightarrow \Lambda$ is a map such that for some $i,j$, $ \eta = \begin{pmatrix} 0 & \cdots & 0 & (e_j)^{(i)} & 0 & \cdots  & 0\end{pmatrix}$ and for all $0\leq p\leq t_{m-n}, o(f^{m-n}_p)\neq e_j$ and $t(f^{m-n}_p)\neq e_j.$
The associated map $\psi_\eta:\K_{m}\xrightarrow{}\K_{m-n+1}$ satisfies
\begin{equation*}
(d\psi_\eta - (-1)^{n-1}\psi_\eta d )(\varepsilon^m_r) = 0.
\end{equation*}
If there are $p'$ such that $o(f^{m-n}_{p'})=e_j$ and $t(f^{m-n}_{p'})=e_j,$ the above equation holds provided $c_{i,p'}(m,r,n)  =  (-1)^{n(m-n)} c_{p',i}(m,r,m-n)$
\end{theo}

\begin{proof}
The comultiplication on the resolution $\K$ is given by \\
$\displaystyle \Delta_{\K}(\varepsilon^m_r) = \sum_{v=0}^{m}\sum_{p=0}^{t_v}\sum_{q=0}^{t_{m-v}}c_{p,q}(m,r,v)\varepsilon^v_p\ot_{\Lambda}\varepsilon^{m-v}_q.$ The right hand side of Equation \eqref{homo-defi} which is $(\eta\ot 1 - 1\ot \eta )\Delta_{\K} (\varepsilon^m_r)$  therefore becomes 
\begin{align*}
&(\eta\ot 1 - 1\ot \eta )\sum_{v=0}^{m}\sum_{p=0}^{t_v}\sum_{q=0}^{t_{m-v}}c_{p,q}(m,r,v)\varepsilon^v_p\ot_{\Lambda}\varepsilon^{m-v}_q\\
 &= \sum_{v=0}^{m}\sum_{p=0}^{t_v}\sum_{q=0}^{t_{m-v}}c_{p,q}(m,r,v)(\eta\ot 1)(\varepsilon^v_p\ot_{\Lambda}\varepsilon^{m-v}_q) -\sum_{v=0}^{m}\sum_{p=0}^{t_v}\sum_{q=0}^{t_{m-v}}c_{p,q}(m,r,v) (1\ot \eta )(\varepsilon^v_p\ot_{\Lambda}\varepsilon^{m-v}_q).
\end{align*}
Whenever $v=n, p=i$ in the first part and $m-v=n, q=i$ in the second part, the above expression yields
\begin{align*}
& \sum_{q=0}^{t_{m-n}}c_{i,q}(m,r,n) (\eta\ot 1 )(\varepsilon^n_i\ot_{\Lambda}\varepsilon^{m-n}_q)  - \sum_{p=0}^{t_{m-n}}c_{p,i}(m,r,m-n) (1\ot \eta )(\varepsilon^{m-n}_p\ot_{\Lambda}\varepsilon^{n}_i)\\
&= \sum_{q=0}^{t_{m-n}}c_{i,q}(m,r,n)\eta(\varepsilon^n_i)\varepsilon^{m-n}_q -  (-1)^{n(m-n)}\sum_{p=0}^{t_{m-n}}c_{p,i}(m,r,m-n)\varepsilon^{m-n}_p\eta(\varepsilon^{n}_i)\\
&= \sum_{q=0}^{t_{m-n}}c_{i,q}(m,r,n) e_j\varepsilon^{m-n}_q -  (-1)^{n(m-n)}\sum_{p=0}^{t_{m-n}}c_{p,i}(m,r,m-n)\varepsilon^{m-n}_p e_j.
\end{align*}
The above expression is equal to $0$ if for every $p$, $o(f^{m-n}_q)\neq e_j$ and $t(f^{m-n}_p)\neq e_j$. Now let $q' \;(0\leq q'\leq t_{m-n})$ be such that $o(f^{m-n}_{q'})=e_j$, and $p' \;(0\leq p'\leq t_{m-n})$ be such that $t(f^{m-n}_{p'})=e_j$ then  $e_j\varepsilon^{m-n}_{q'}$ is equal to 
\begin{align*}
&e_j( 0,\cdots, 0, o(f^{m-n}_{q'})\ot_k t(f^{m-n}_{q'}), 0,\cdots, 0)= ( 0,\cdots, 0, e_jo(f^{m-n}_{q'})\ot_k t(f^{m-n}_{q'}), 0,\cdots, 0)\\
&= ( 0,\cdots, 0, e_j^2\ot_k t(f^{m-n}_{q'}), 0,\cdots, 0)= ( 0,\cdots, 0, e_j\ot_k t(f^{m-n}_{q'}), 0,\cdots, 0)\\
&= ( 0,\cdots, 0, o(f^{m-n}_{q'})\ot_k t(f^{m-n}_{q'}), 0,\cdots, 0) = \varepsilon^{m-n}_{q'}
\end{align*}
and similarly $\varepsilon^{m-n}_{q'}e_j = \varepsilon^{m-n}_{p'}.$ The last expression of the right hand side of Equation \eqref{homolifting} therefore becomes
\begin{align*}
&= \sum_{q'}c_{i,q'}(m,r,n) \varepsilon^{m-n}_{q'} -  (-1)^{n(m-n)}\sum_{p'}c_{p',i}(m,r,m-n)\varepsilon^{m-n}_{p'}\\
&= \sum_{p}[c_{i,p}(m,r,n) -  (-1)^{n(m-n)}c_{p,i}(m,r,m-n)]\varepsilon^{m-n}_{p} = 0
\end{align*}
for some $p$, $0\leq p\leq t_{m-n}$ for which $o(f^{m-n}_p)=e_j$ and $t(f^{m-n}_p)=e_j.$
\end{proof}

\begin{rema}\label{rema-koszul-idem}
A special case of Theorem \ref{homotopylifting00} occurs if $\eta :\K_n\rightarrow \Lambda$ is a cocycle and $m=n+1$. The associated map $\psi_\eta:\K_{m}\xrightarrow{}\K_{m-n+1}$ satisfies $(d\psi_\eta - (-1)^{n-1}\psi_\eta d )(\varepsilon^m_r) = 0$ and it is a homotopy lifting of $\eta.$ This means that $\eta(\varepsilon^n_i)=e_j$ and $\eta(\varepsilon^n_r)=0$ for any $r\neq i$ and as a cocycle $0 = d^{*}\eta(\varepsilon^{n+1}_r) = \eta d(\varepsilon^{n+1}_r)$ which is equal to
\begin{align*}
& \eta\sum_{j=0}^{t_n}\Big( \sum_{p=0}^{t_{1}}c_{p,j}(n+1,r,1)f^1_p\varepsilon^n_j + (-1)^{n+1} \sum_{q=0}^{t_{1}}c_{j,q}(n+1,r,n)\varepsilon^{n}_j f^1_q\Big)\\
&= \sum_{p=0}^{t_{1}}c_{p,i}(n+1,r,1)f^1_p\eta(\varepsilon^n_i) + (-1)^{n+1} \sum_{q=0}^{t_{1}}c_{i,q}(n+1,r,n)\eta(\varepsilon^{n}_i) f^1_q\\
&= \sum_{p=0}^{t_{1}}c_{p,i}(n+1,r,1)f^1_pe_j + (-1)^{n+1} \sum_{q=0}^{t_{1}}c_{i,q}(n+1,r,n)e_j f^1_q.
\end{align*}
There are some paths having origin and terminal vertex as $e_j$. Suppose such paths are $f^1_{p'}$, the above expression becomes $\displaystyle{ \sum_{p'}[c_{p',i}(n+1,r,1) + (-1)^{n+1} c_{i,p'}(n+1,r,n) ]f^1_{p'}}$ and hence $c_{p',i}(n+1,r,1) = (-1)^{n} c_{i,p'}(n+1,r,n)$ for some $0\leq p'\leq t_1$.

We recall from Definition \ref{homolift} that for Koszul algebras, we can take the first homotopy lifting map $(\psi_{\eta})_{n-1}:\K_{n-1}\xrightarrow{}\K_0$ to be the zero map. From the result of Theorem \ref{homotopylifting00}, $d(\psi_\eta)_{n} = (-1)^{n-1}(\psi_\eta)_{n-1} d = 0,$ so we see that we can define all homotopy lifting maps $\psi_\eta$ to be the zero map for all $n$.
\end{rema}

\begin{coro}
Let $\Lambda = kQ/I$ be a Koszul algebra and suppose that $\eta :\K_n\rightarrow \Lambda$ is a cocycle such that for some $i,j$,  $ \eta = \begin{pmatrix} 0 & \cdots & 0 & (e_j)^{(i)} & 0 & \cdots  & 0\end{pmatrix}.$ Then the associated homotopy lifting maps to $\eta$ are all zero.
\end{coro}

Moving on to the case where a free basis element is mapped to a path of length 1. We start with the following definition.

\begin{defi}
For each fixed $n,r$, let $<\;\;>_{n,r}:\K_{n-1}\rightarrow \K_{n-1}$ be a map defined on $\varepsilon^{n-1}_j$ for each $j$ by 
\begin{equation*}\label{big-rangle}
<\varepsilon^{n-1}_j>_{n,r}= \sum_{v=0}^{t_{n-1}}\Big(\sum_{p=0}^{t_1}w^{(j)}_{pv}(n,r,1)f_p^1 \varepsilon^{n-1}_v 
+  \sum_{q=0}^{t_1}w^{(j)}_{vq}(n,r,n-1)\varepsilon^{n-1}_v f_q^1\Big)
\end{equation*}
for scalars $w^{(j)}_{pv}(n,r,1)$ and $w^{(j)}_{vq}(n,r,n-1)$. Then extend to all of $\K_{n-1}$ by requiring it to be a $\Lambda^e$-module homomorphism.
\end{defi}
\begin{rema}
Whenever $w^{(j)}_{pv}(n,r,1)=0=w^{(j)}_{vq}(n,r,n-1)$ for all $v\neq j$ and $w^{(j)}_{pj}(n,r,1)=c_{pj}(n,r,1), \; w^{(j)}_{jq}(n,r,n-1) = (-1)^n c_{jq}(n,r,n-1)$, we obtain a special case of the module homomorphism which is defined as 
\begin{equation}\label{rangle}
<\varepsilon^{n-1}_j>_{n,r}= \sum_{p=0}^{t_1}c_{pj}(n,r,1)f_p^1 \varepsilon^{n-1}_j 
+ (-1)^n \sum_{q=0}^{t_1}c_{jq}(n,r,n-1)\varepsilon^{n-1}_j f_q^1.
\end{equation}
The differential factors through this map i.e. there is a component of the differential map $d^j_{n}$, $j=0,1,\ldots,t_{n-1}$ taking every basis element $\varepsilon^n_r$ to a free basis element $\varepsilon^{n-1}_j$ such that in the special case defined above, the following diagram
\[
\xymatrix{
\K_{n} \ar[r]^{\;d^j_n\;} \ar[dr]_{d_n} & \K_{n-1} \ar[d]^{<\;>_{n,r}} \\
 & \K_{n-1}
}
\]
commutes i.e. $<d^j_n(\varepsilon^n_r)>:= <\varepsilon^{n-1}_j>_{n,r}$ and $d(\varepsilon^n_r) = \sum_{j=0}^{t_{n-1}}<d^j_n(\varepsilon^n_r)>$. The subscript $(n,r)$ in $<\;\;\;>_{n,r}$ indicates that the scalars $c_{pj}(n,r,1)$ and $(-1)^nc_{jq}(n,r,n-1)$ are coming from or associated with the basis element $\varepsilon^n_r$. For example, in the expansion of $<\varepsilon^{n-1}_j>_{n,r}$, $<\varepsilon^{n-1}_{j+1}>_{n,r}$ and $<\varepsilon^{n-1}_j>_{n,r+1},$  the associated scalars to $<\varepsilon^{n-1}_j>_{n,r}$ will be $c_{pj}(n,r,1)$ and $(-1)^nc_{jq}(n,r,n-1)$ and they come from the expansion of $d_n(\varepsilon^n_r )$, the associated scalars to $<\varepsilon^{n-1}_{j+1}>_{n,r}$ will be $c_{p,j+1}(n,r,1)$ and $(-1)^nc_{j+1,q}(n,r,n-1)$ and they come from the expansion of $d_n(\varepsilon^n_r )$, while the associated scalars to $<\varepsilon^{n-1}_j>_{n,r+1}$ will be $c_{pj}(n,r+1,1)$ and $(-1)^nc_{jq}(n,r+1,n-1)$ and they come from the expansion of $d_n(\varepsilon^n_{r+1} ).$ We immediately see that under these conditions, 
$$ d_n(\varepsilon^n_r ) = \sum_{j=0}^{t_{n-1}}<\varepsilon^{n-1}_j>_{n,r} = \sum_{j=0}^{t_{n-1}}\Big( \sum_{p=0}^{t_1}c_{pj}(n,r,1)f_p^1 \varepsilon^{n-1}_j + (-1)^n \sum_{q=0}^{t_1}c_{jq}(n,r,n-1)\varepsilon^{n-1}_j f_q^1 \Big) .$$
Henceforth, we will make use of the special case module homomorphism $<\;\;>_{n,r}$ because of its connection to the differentials $d$.
\end{rema}

We now give another main result about cocycles taking free basis elements to paths of length 1. We assume in the proof of the next results, that the scalars $c_{pj}(n,r,1),c_{jq}(n,r,n-1)$ of the comultiplicative structure~\eqref{comult-struc} satisfy \textbf{Equation \ref{main-result}} below:
\begin{align}\label{main-result}
c_{pj}(m,r,1) &= c_{pj'}(m-n+1,r,1) = c_{p,j'+1}(m-n+1,r,1) \quad \text{and}\notag\\
(-1)^{m} c_{jq}(m,r,m-1) &= (-1)^{m-n+1}c_{j'q}(m-n+1,r,m-n)\notag\\ 
& =  (-1)^{m-n+1}c_{j'+1,q}(m-n+1,r,m-n).
\end{align}
In the next section where we will show several examples, these scalars are mostly $\pm1$, powers of $\pm q$ or 0. However, they can assume other values different from these in general.  For an $n$ cocycle $\eta$, a homotopy lifting map of $\eta$ is a map $\psi_\eta:\K_m\xrightarrow{}\K_{m-n+1}$ for any $m$. Our goal is to define $\psi_\eta$ such that Equation \eqref{homo-defi} holds. Now define it in the following way:
$$\psi_\eta (\varepsilon^m_r) = \sum_{j=0}^{t_{m-n+1}} b_{m,r}(m-n+1,j)\varepsilon^{m-n+1}_j,$$
where $b_{m,r}(m-n+1,j)$ are scalars and extend it to all of $\K_m$ by requiring it to be a $\Lambda^e$-module homomorphism. Now consider the special case where $b_{m,r}(m-n+1,j)=0$ for all $j\neq s$ that is $\psi_\eta (\varepsilon^m_r) = b_{m,r}(m-n+1,s)\varepsilon^{m-n+1}_s$ for some $s$. The next series of results shows that this special case is indeed a homotopy lifting map under certain conditions on the scalars $b_{m,r}(m-n+1,s)$ only. We start with the following lemma.
\begin{lemma}\label{contralem1}
Suppose $\eta :\K_n\rightarrow \Lambda$ is a cocycle. The $\Lambda^e$-module map defined by $\psi_\eta (\varepsilon^m_r) = b_{m,r}(m-n+1,s)\varepsilon^{m-n+1}_s$ satisfies $\psi_\eta(<\varepsilon^{m-1}_j>_{m,r}) = <\psi_\eta(\varepsilon^{m-1}_j)>_{m,r}.$ Furthermore, whenever  Equation~\eqref{main-result} holds, 
\begin{align*}
b_{m-1,j}(m-n,j')<\varepsilon^{m-n}_{j'}>_{m,r} = b_{m-1,j}(m-n,j')<\varepsilon^{m-n}_{j'}>_{m-n+1,r}.
\end{align*}
\end{lemma}
\begin{proof}
We first observe that $f^1_w \psi_\eta (\varepsilon^m_r) = f^1_w b_{m,r}(m-n+1,s) \varepsilon^{m-n+1}_s $ for any $0\leq w\leq t_1$. This is the same as $b_{,m,r}(m-n+1,s) f^1_w\varepsilon^{m-n+1}_s = \psi_\eta (f^1_w \varepsilon^{m-n+1}_s)$ since $\psi_\eta$ is a $\Lambda^e$-module homomorphism. $  \psi_\eta (\varepsilon^m_r) f^1_w = \psi_\eta (\varepsilon^m_r f^1_w )$ holds similarly.
Taking $\psi_\eta (\varepsilon^{m-1}_j) =  b_{m-1,j}(m-n,j') \varepsilon^{m-n}_{j'}$, we will have
\begin{align*}
& \psi_\eta(<\varepsilon^{m-1}_j>_{m,r})= \psi_\eta\Big( \sum_{p=0}^{t_1}c_{pj}(m,r,1)f_p^1 \varepsilon^{m-1}_j 
+ (-1)^m \sum_{q=0}^{t_1}c_{jq}(m,r,m-1)\varepsilon^{m-1}_j f_q^1 \Big)\\
&=  \sum_{p=0}^{t_1}c_{pj}(m,r,1) f_p^1 \psi_\eta (\varepsilon^{m-1}_j) 
+ (-1)^m \sum_{q=0}^{t_1}c_{jq}(m,r,m-1)\psi_\eta(\varepsilon^{m-1}_j )f_q^1 \\
&=  \sum_{p=0}^{t_1}c_{pj}(m,r,1) b_{m-1,j}(m-n,j') f_p^1 \varepsilon^{m-n}_{j'} \\
&+ (-1)^m \sum_{q=0}^{t_1}c_{jq}(m,r,m-1) b_{m-1,j}(m-n,j') \varepsilon^{m-n}_{j'} f_q^1.
\end{align*}
On the other hand
\begin{align*}
<\psi_\eta(\varepsilon^{m-1}_j)>_{m,r} &= <b_{m-1,j}(m-n,j')\varepsilon^{m-n}_{j'} >_{m,r}\\
&= \sum_{p=0}^{t_1}c_{pj}(m,r,1)  b_{m-1,j}(m-n,j') f_p^1 \varepsilon^{m-n}_{j'}\\
&+ (-1)^m \sum_{q=0}^{t_1}c_{jq}(m,r,m-1) b_{m-1,j}(m-n,j') \varepsilon^{m-n}_{j'} f_q^1 .
\end{align*}
It is now established that $\psi_\eta(<\varepsilon^{m-1}_j>_{m,r}) = <\psi_\eta(\varepsilon^{m-1}_j)>_{m,r}$.  Also, you can factor out the scalars $b_{m-1,j}(m-n,j')$ from the last expansion so that \\
$\displaystyle{\psi_\eta(<\varepsilon^{m-1}_j>_{m,r}) = b_{m-1,j}(m-n,j') <\varepsilon^{m-n}_{j'} >_{m,r}}$. From the expansion of \\ $\displaystyle{(<\varepsilon^{m-n}_{j'}>_{m-n+1,r}) }$ the equality $\psi_\eta(<\varepsilon^{m-1}_j>_{m,r}) = b_{m-1,j}(m-n,j') <\varepsilon^{m-n}_{j'} >_{m,r} = b_{m-1,j}(m-n,j')<\varepsilon^{m-n}_{j'} >_{m-n+1,r}$ holds provided that Equation~\eqref{main-result} is satisfied: that is whenever $c_{pj}(m,r,1) = c_{pj'}(m-n+1,r,1)$ and $(-1)^{m} c_{jq}(m,r,m-1) = (-1)^{m-n+1}c_{j'q}(m-n+1,r,m-n)$.  Moreover, since the map $\psi_\eta$ maps basis elements $\{\varepsilon^{m-1}_r\}_{r=0}^{t_{m-1}}$ of $\K_{m-1}$ to basis elements $\{\varepsilon^{m-n}_r\}_{r=0}^{t_{m-n}}$ of $\K_{m-n}$, over a sum, the index $j$ shifts to $j'$ where $0\leq j'\leq t_{m-n}$. This means that
\begin{equation}\label{psipsi}
\sum_{j=0}^{t_{m-1}}\psi_\eta(<\varepsilon^{m-1}_j>_{m,r}) = \sum_{j=0}^{t_{m-1}}<\psi_\eta(\varepsilon^{m-1}_j)>_{m,r} = \sum_{j'=0}^{t_{m-n}} b_{m-1,j}(m-n,j')<\varepsilon^{m-n}_{j'} >_{m-n+1,r}.
\end{equation}
\end{proof}

We now give the first main result.
\begin{theo}\label{homotopylifting1}
Let $\Lambda = kQ/I$ be a Koszul algebra. Suppose that $\eta :\K_n\rightarrow \Lambda$ is a cocycle such that   $ \eta = \begin{pmatrix} 0 & \cdots & 0 & (f^1_w)^{(i)} & 0 & \cdots  & 0\end{pmatrix}$ for some $0\leq w\leq t_1$.
A homotopy lifting map  $\psi_\eta :\K_{m} \rightarrow \K_{m-n+1}$ associated to $\eta$ can be defined by
$$\psi_\eta (\varepsilon^m_r) = b_{m,r}(m-n+1,s)\varepsilon^{m-n+1}_s$$
for some scalars $b_{m,r}(m-n+1,s)$. Moreover the scalars $b_{m,r}(m-n+1,s)$ satisfy
\begin{itemize} 
\item[(i).] $B = \begin{cases} c_{i,\alpha}(m,r,1) &{\rm when }\; p=w \\ 0 &{\rm when }\; p\neq w \end{cases}$ \qquad and
\item[(ii).] $B' = \begin{cases} (-1)^{n(m-n)+1}c_{\alpha,i}(m,r,m-n) &{\rm when }\; p=w \\ 0 &{\rm when }\; p\neq w, \end{cases}$
\end{itemize}
for all $\alpha$, where 
\begin{multline*}
B = b_{m,r}(m-n+1,s) c_{p\alpha}(m-n+1,s,1) +(-1)^n b_{m-1,j}(m-n,\alpha) c_{p\alpha}(m-n+1,r,1), \text{ and }\\
B' =  (-1)^{m+1} [(-1)^{n}b_{m,r}(m-n+1,s) c_{\alpha q}(m-n+1,s,m-n) \\+ b_{m-1,j}(m-n,\alpha) c_{\alpha q}(m-n+1,r,m-n)].
\end{multline*}
\end{theo}
\begin{proof}
We have to show that under the stated conditions (i) and (ii), the equation
\begin{equation*}\label{homocontra1}
(d\psi_\eta - (-1)^{n-1}\psi_\eta d )(\varepsilon^m_r) = (\eta\ot 1 - 1\ot\eta)\Delta_{\K} (\varepsilon^m_r)
\end{equation*}
holds. Again, we use the Koszul sign convention in the expansion of $(1\ot\eta)(\varepsilon^n_r\ot\varepsilon^m_s)$ to obtain $(-1)^{|\eta|n}\varepsilon^n_r\cdot\eta(\varepsilon^m_s)$, where $|\eta|$ is the degree of $\eta$. Next, we have that
\begin{align*}
&(d\psi_\eta - (-1)^{n-1}\psi_\eta d) (\varepsilon^m_r) = d\psi_\eta (\varepsilon^m_r) - (-1)^{n-1} \psi_\eta d (\varepsilon^m_r) \\
&= d(b_{m,r}(m-n+1,s)\varepsilon^{m-n+1}_s) -  (-1)^{n-1}\psi_\eta( \sum_{j=0}^{t_{m-1}}<\varepsilon^{m-1}_j>_{m,r})\\
&= b_{m,r}(m-n+1,s) d(\varepsilon^{m-n+1}_s)  - (-1)^{n-1} \sum_{j=0}^{t_{m-1}}\psi_\eta(<\varepsilon^{m-1}_j>_{m,r}).
\end{align*}
Using Equation~\eqref{psipsi}, and applying the definition of the differential, we get 
\begin{multline*}
b_{m,r}(m-n+1,s) \sum_{\alpha=0}^{t_{m-n}}<\varepsilon^{m-n}_\alpha>_{m-n+1,s} - (-1)^{n-1} \sum_{j'=0}^{t_{m-n}}b_{m-1,j}(m-n,j') <\varepsilon^{m-n}_{j'}>_{m-n+1,r},
\end{multline*}
then applying the definition of the map $<\;\cdot\;>_{*,*}$ 
\begin{align*}
& = \sum_{\alpha=0}^{t_{m-n}} \sum_{p=0}^{t_1} b_{m,r}(m-n+1,s) c_{p\alpha}(m-n+1,s,1)f_p^1 \varepsilon^{m-n}_\alpha \\
&+ (-1)^{m-n+1}  \sum_{\alpha=0}^{t_{m-n}}\sum_{q=0}^{t_1} b_{m,r}(m-n+1,s) c_{\alpha q}(m-n+1,s,m-n)\varepsilon^{m-n}_\alpha f_q^1\\
& + (-1)^{n} \sum_{j'=0}^{t_{m-n}} \sum_{p=0}^{t_1} b_{m-1,j}(m-n,j') c_{pj'}(m-n+1,r,1)f_p^1 \varepsilon^{m-n}_{j'} \\
&+ (-1)^{m+1}  \sum_{j'=0}^{t_{m-n}}\sum_{q=0}^{t_1} b_{m-1,j}(m-n,j') c_{j'q}(m-n+1,r,m-n)\varepsilon^{m-n}_{j'} f_q^1.
\end{align*}
Collecting like terms and re-indexing, we get
\begin{align*}
& \sum_{\alpha=0}^{t_{m-n}} \sum_{p=0}^{t_1} \Big[ b_{m,r}(m-n+1,s) c_{p\alpha}(m-n+1,s,1) \\
&+(-1)^n b_{m-1,j}(m-n,\alpha) c_{p\alpha}(m-n+1,r,1)\Big]f_p^1 \varepsilon^{m-n}_\alpha \\
&+(-1)^{m+1} \sum_{\alpha=0}^{t_{m-n}} \sum_{p=0}^{t_1} \Big[ (-1)^n b_{m,r}(m-n+1,s) c_{\alpha p}(m-n+1,s,m-n) \\
&+ b_{m-1,j}(m-n,\alpha) c_{\alpha p}(m-n+1,r,m-n)\Big] \varepsilon^{m-n}_\alpha f_p^1
\end{align*}
which is succintly expressed as
\begin{align*}
& = \sum_{\alpha=0}^{t_{m-n}} \sum_{p=0}^{t_1} \Big[ B\Big]f_p^1 \varepsilon^{m-n}_\alpha + \sum_{\alpha=0}^{t_{m-n}} \sum_{p=0}^{t_1} \Big[ B' \Big] \varepsilon^{m-n}_\alpha f_p^1.
\end{align*}
After applying the definitions of $B$ and $B'$ given by (i) and (ii) of the theorem, that is substitute $B= c_{i,\alpha}(m,r,1)$ when $p=w$ and $B'= -(-1)^{n(m-n)}c_{\alpha,i}(m,r,m-n)$ and $0$ otherwise, we get
\begin{align*}
&= \sum_{\alpha=0}^{t_{m-n}}c_{i,\alpha}(m,r,1) f^1_w\varepsilon^{m-n}_\alpha -  (-1)^{n(m-n)}\sum_{\alpha=0}^{t_{m-n}}c_{\alpha,i}(m,r,m-n)\varepsilon^{m-n}_{\alpha} f^1_w.
\end{align*}
On the other hand, the comultiplication on the resolution $\K$ is given by \\
$\displaystyle \Delta_{\K}(\varepsilon^m_r) = \sum_{v=0}^{m}\sum_{p=0}^{t_v}\sum_{q=0}^{t_{m-v}}c_{p,q}(m,r,v)\varepsilon^v_p\ot_{\Lambda}\varepsilon^{m-v}_q.$ Applying $(\eta\ot 1 - 1\ot \eta )$, we obtain
\begin{align*}
&(\eta\ot 1 - 1\ot \eta )\Delta_{\K} (\varepsilon^m_r) = (\eta\ot 1 - 1\ot \eta )\sum_{v=0}^{m}\sum_{p=0}^{t_v}\sum_{q=0}^{t_{m-v}}c_{p,q}(m,r,v)\varepsilon^v_p\ot_{\Lambda}\varepsilon^{m-v}_q\\
 &= \sum_{v=0}^{m}\sum_{p=0}^{t_v}\sum_{q=0}^{t_{m-v}}c_{p,q}(m,r,v)(\eta\ot 1)(\varepsilon^v_p\ot_{\Lambda}\varepsilon^{m-v}_q)\\
 &-\sum_{v=0}^{m}\sum_{p=0}^{t_v}\sum_{q=0}^{t_{m-v}}c_{p,q}(m,r,v) (1\ot \eta )(\varepsilon^v_p\ot_{\Lambda}\varepsilon^{m-v}_q).
\end{align*}
Whenever $v=n, p=i$ in the first part and $m-v=n, q=i$ in the second part, and changing the indices $p,q$ to $\alpha$ later on, the above expression yields
\begin{align*}
& \sum_{q=0}^{t_{m-n}}c_{i,q}(m,r,1) (\eta\ot 1 )(\varepsilon^n_i\ot_{\Lambda}\varepsilon^{m-n}_q)  - \sum_{p=0}^{t_{m-n}}c_{p,i}(m,r,m-n) (1\ot \eta )(\varepsilon^{m-n}_p\ot_{\Lambda}\varepsilon^{n}_i)\\
&= \sum_{q=0}^{t_{m-n}}c_{i,q}(m,r,1)\eta(\varepsilon^n_i)\varepsilon^{m-n}_q -  (-1)^{n(m-n)}\sum_{p=0}^{t_{m-n}}c_{p,i}(m,r,m-n)\varepsilon^{m-n}_p\eta(\varepsilon^{n}_i)\\
&= \sum_{\alpha=0}^{t_{m-n}}c_{i,\alpha}(m,r,1) f^1_w\varepsilon^{m-n}_{\alpha} -  (-1)^{n(m-n)}\sum_{\alpha=0}^{t_{m-n}}c_{\alpha,i}(m,r,m-n)\varepsilon^{m-n}_\alpha f^1_w \\
& = (d\psi_\eta - \psi_\eta d) (\varepsilon^n_r).
\end{align*}
\end{proof}

We will next consider the case where free basis elements of $\K_m$ are mapped to paths of length 2. We start with the following definition.
\begin{defi}
Let $f^1_{w},f^1_{w+1}$ be paths of length 1 in $\Lambda = kQ/I$. For any $m,n$, let $0\leq r\leq t_m$, $0\leq s\leq t_{m-n+1},$ and define a map  $\psi:\K_m \rightarrow \K_{m-n+1},$ on the free basis elements $\{\varepsilon^m_r\}_{r=0}^{t_m}$ of $\K_m$ by
\begin{equation}\label{general-homo2}
\psi (\varepsilon^m_r) = \sum_{v=0}^{t_{m-n+1}-1}b_{m,r}(m-n+1,v+1)f^1_{w}\varepsilon^{m-n+1}_{v+1} + b_{m,r}(m-n+1,v)\varepsilon^{m-n+1}_v f^1_{w+1},
\end{equation}
for some scalars $b_{m,r}(m-n+1,s+1)$ and $b_{m,r}(m-n+1,s)$ and extend it to all of $\K_{m}$ as a $\Lambda^e$-module homomorphism. 
\end{defi}
\begin{rema}
For a cocycle $\eta :\K_n\rightarrow \Lambda$, Y. Volkov showed that there are homotopy lifting maps $\psi_\eta:\K\rightarrow \K[1-n],$ as presented in Definition~\ref{homolift}. We will show that under certain conditions on the scalars, the above map is a homotopy lifting map for some cocycle $\eta$ such that $\eta(\varepsilon^n_i) = f^1_w f^1_{w+1}$ for some $i,w$, and $\eta(\varepsilon^n_j)=0$ when $i\neq j$ (see Theorem~\ref{homotopylifting2} for instance). For a fixed index $s$, If $b_{m,r}(m-n+1,v+1) = b_{m,r}(m-n+1,v) = 0$ for all $v\neq s$, we obtain the following special case:
$$\psi_\eta (\varepsilon^m_r) = b_{m,r}(m-n+1,s+1)f^1_{w}\varepsilon^{m-n+1}_{s+1} + b_{m,r}(m-n+1,s)\varepsilon^{m-n+1}_s f^1_{w+1}$$
and show that under certain conditions on the scalars, this is a homotopy lifting map for $\eta$.
\end{rema}

Before presenting another major theorem (Theorem~\ref{homotopylifting2}), we present two lemmas. Since we will be expanding $\psi_{\eta}(<\varepsilon^m_r>_{m,r})$, the first lemma gives information on how this expansion turns out. The second lemma helps to give a succinct way to express the sum $d\psi_{\eta}(\varepsilon^m_r) + \psi_{\eta}d(\varepsilon^m_r)$ in case $\eta$ takes a basis element to a path of length 2.
\begin{lemma}\label{contralem2}
Let $\Lambda = kQ/I$ be a Koszul algebra. Suppose $\eta :\K_n\rightarrow \Lambda$ is a cocycle. The map defined by $$\psi_\eta (\varepsilon^m_r) = b_{m,r}(m-n+1,s+1)f^1_{w}\varepsilon^{m-n+1}_{s+1} + b_{m,r}(m-n+1,s)\varepsilon^{m-n+1}_s f^1_{w+1}$$ satisfies $\psi_\eta(<\varepsilon^{m-1}_j>_{m,r}) = <\psi_\eta(\varepsilon^{m-1}_j)>_{m,r}$. Moreover, the expansion
\begin{align*}
<\psi_\eta(\varepsilon^{m-1}_j)>_{m,r} &= b_{m-1,j}(m-n,j'+1)<f^1_w\varepsilon^{m-n}_{j'+1}>_{m-n+1,r} \\
&+  b_{m-1,j}(m-n,j')<\varepsilon^{m-n}_{j'} f^1_{w+1}>_{m-n+1,r}
\end{align*}
holds whenever  Equation~\eqref{main-result} holds. 
\end{lemma}
\begin{proof}
We use the fact that $\psi_\eta$ and $<\;\cdot\;>_{m,r}$ are $\Lambda^e$-module homomorphism.
Taking $\psi_{\eta} (\varepsilon^{m-1}_j) = b_{m-1,j}(m-n,j'+1)f^1_{w}\varepsilon^{m-n}_{j'+1} + b_{m-1,j}(m-n,j')\varepsilon^{m-n}_{j'} f^1_{w+1}$, we will have
\begin{align*}
& \psi_\eta(<\varepsilon^{m-1}_j>_{m,r})\\
&= \psi_\eta\Big( \sum_{p=0}^{t_1}c_{pj}(m,r,1)f_p^1 \varepsilon^{m-1}_j 
+ (-1)^m \sum_{q=0}^{t_1}c_{jq}(m,r,m-1)\varepsilon^{m-1}_j f_q^1 \Big)\\
&=  \sum_{p=0}^{t_1}c_{pj}(m,r,1)\psi_\eta (f_p^1 \varepsilon^{m-1}_j) 
+ (-1)^m \sum_{q=0}^{t_1}c_{jq}(m,r,m-1)\psi_\eta(\varepsilon^{m-1}_j f_q^1 )\\
&=  \sum_{p=0}^{t_1}c_{pj}(m,r,1) f_p^1 \psi_\eta (\varepsilon^{m-1}_j) 
+ (-1)^m \sum_{q=0}^{t_1}c_{jq}(m,r,m-1)\psi_\eta(\varepsilon^{m-1}_j )f_q^1 = <\psi_\eta(\varepsilon^{m-1}_j)>_{m,r}
\end{align*}
\begin{align*}
&=  \sum_{p=0}^{t_1}c_{pj}(m,r,1) f_p^1 \big[ b_{m-1,j}(m-n,j'+1)f^1_{w}\varepsilon^{m-n}_{j'+1} + b_{m-1,j}(m-n,j')\varepsilon^{m-n}_{j'} f^1_{w+1}\big] \\
&+ (-1)^m \sum_{q=0}^{t_1}c_{jq}(m,r,m-1) \big[ b_{m-1,j}(m-n,j'+1)f^1_{w}\varepsilon^{m-n}_{j'+1} + b_{m-1,j}(m-n,j')\varepsilon^{m-n}_{j'} f^1_{w+1}\big] f_q^1
\end{align*}
\begin{align*}
&=  b_{m-1,j}(m-n,j'+1)\Big[\sum_{p=0}^{t_1}c_{pj}(m,r,1) f_p^1 (f^1_{w}\varepsilon^{m-n}_{j'+1}) + (-1)^m \sum_{q=0}^{t_1}c_{jq}(m,r,m-1)  (f^1_{w}\varepsilon^{m-n}_{j'+1}) f^1_q \Big]\\
&+  b_{m-1,j}(m-n,j')\Big[\sum_{p=0}^{t_1}c_{pj}(m,r,1) f_p^1 (\varepsilon^{m-n}_{j'}f^1_{w+1}) + (-1)^m \sum_{q=0}^{t_1}c_{jq}(m,r,m-1)  (\varepsilon^{m-n}_{j'}f^1_{w+1}) f^1_q \Big].
\end{align*}
We now recall that the first equality of Equation~\eqref{main-result} implies that 
\begin{align*}
&<\varepsilon^{m-n}_{j'} f^1_{w+1}>_{m-n+1,r} = \sum_{p=0}^{t_1}c_{pj'}(m-n+1,r,1)f_p^1 (\varepsilon^{m-n}_{j'}f^1_{w+1}) \\
&+ (-1)^{m-n+1} \sum_{q=0}^{t_1}c_{j'q}(m-n+1,r,m-n)(\varepsilon^{m-n}_{j'}f^1_{w+1}) f_q^1 \\
&= \sum_{p=0}^{t_1}c_{pj}(m,r,1)f_p^1 (\varepsilon^{m-n}_{j'}f^1_{w+1}) + (-1)^{m} \sum_{q=0}^{t_1}c_{jq}(m,r,m-1)(\varepsilon^{m-n}_{j'}f^1_{w+1}) f_q^1,
\end{align*}
and the second part of the equality of Equation~\eqref{main-result} i.e. $c_{pj}(m,r,1) = c_{p,j'+1}(m-n+1,r,1)$ and $(-1)^{m} c_{jq}(m,r,m-1) = (-1)^{m-n+1}c_{j'+1,q}(m-n+1,r,m-n)$ implies that 
\begin{align*}
&<f^1_{w}\varepsilon^{m-n}_{j'+1} >_{m-n+1,r} = \sum_{p=0}^{t_1}c_{p,j'+1}(m-n+1,r,1)f_p^1 (f^1_{w}\varepsilon^{m-n}_{j'+1}) \\
&+ (-1)^{m-n+1} \sum_{q=0}^{t_1}c_{j'+1,q}(m-n+1,r,m-n)(f^1_{w}\varepsilon^{m-n}_{j'+1}) f_q^1 \\
&= \sum_{p=0}^{t_1}c_{pj}(m,r,1)f_p^1 (f^1_{w}\varepsilon^{m-n}_{j'+1}) + (-1)^{m} \sum_{q=0}^{t_1}c_{jq}(m,r,m-1)(f^1_{w}\varepsilon^{m-n}_{j'+1}) f_q^1.
\end{align*}
Putting all these together, we get the desired result:
\begin{align*}
&\psi_\eta(<\varepsilon^{m-1}_j>_{m,r}) = <\psi_\eta(\varepsilon^{m-1}_j)>_{m,r}\\
& = <b_{m-1,j}(m-n,j'+1)f^1_{w}\varepsilon^{m-n}_{j'+1} + b_{m-1,j}(m-n,j')\varepsilon^{m-n}_{j'} f^1_{w+1}>_{m-n+1,r}\\
&= b_{m-1,j}(m-n,j'+1)<f^1_w\varepsilon^{m-n}_{j'+1}>_{m-n+1,r} +  b_{m-1,j}(m-n,j')<\varepsilon^{m-n}_{j'} f^1_{w+1}>_{m-n+1,r}.
\end{align*}
\end{proof}

The differentials map free basis elements $\varepsilon^m_r$ to a linear combination of $f^1_{*}\varepsilon^{m-1}_{**}$ and $\varepsilon^{m-1}_{**} f^1_{*}$ and the map $\psi$ of Definition \ref{general-homo2} maps free basis elements $\varepsilon^{m-1}_{*}$ to linear combination of $f^1_{**}\varepsilon^{m-n}_{***}$ and $\varepsilon^{m-n}_{***} f^1_{**}.$ Combining these two maps means $\psi d$ and $d\psi$ will map free basis elements $\varepsilon^m_r$ to a linear combination of $f^1_{*}f^1_{**}\varepsilon^{m-n}_{***}$,$\displaystyle{ f^1_{*}\varepsilon^{m-n}_{***} f^1_{**}}$,
$\displaystyle{ f^1_{**}\varepsilon^{m-n}_{***}f^1_{*}}$ and $\varepsilon^{m-n}_{***} f^1_{**}f^1_{*}.$ The map $J: \K_m\rightarrow \K_{m-n+1}$ given in the next defiinition describes all the possible $k$-linear combinations there are and the next lemma shows that after suitable substitution of certain scalars, $J = d\psi_\eta -(-1)^{n-1} \psi_\eta d.$

\begin{defi}
For $0\leq r\leq t_{m-n}$, define a map $J: \K_m\rightarrow \K_{m-n}$  on the basis elements $\varepsilon^m_s$ of $\K_m$ by
\begin{equation*}
J(\varepsilon^m_s) = \sum_{r=0}^{t_m}\sum_{i=0}^{t_1}\sum_{j=0}^{t_1}\Big[ \sigma_{m,s}(i,j,r) f^1_i f^1_j \varepsilon^{m-n}_r +  \sigma_{m,s}(i,r,j)  f^1_i \varepsilon^{m-n}_r  f^1_j+  \sigma_{m,s}(r,i,j)  \varepsilon^{m-n}_rf^1_i f^1_j\Big]
\end{equation*}
for some scalars $ \sigma_{m,s}(i,j,r)$ and extend to all of $\K_m$ by requiring it to be a $\Lambda^e$-module homomorphism.
\end{defi}
\begin{lemma}\label{contralem3}
Let $\Lambda = kQ/I$ be a Koszul algebra and $\eta :\K_n\rightarrow \Lambda$ a cocycle. There are scalars  $\sigma_{*,*}(*,*,*)$ such that
\begin{equation}\label{d-homo}
(d\psi_\eta - (-1)^{n-1}\psi_\eta d )(\varepsilon^m_r) = J(\varepsilon^m_r).
\end{equation}
\end{lemma}
\begin{proof}
We begin the proof with direct evaluation of these maps on the free basis elements.
\begin{align*}
&(d\psi_\eta - (-1)^{n-1}\psi_\eta d) (\varepsilon^m_r) = d\psi_\eta (\varepsilon^m_r) - (-1)^{n-1} \psi_\eta d (\varepsilon^m_r) \\
&= d\Big(b_{m,r}(m-n+1,s+1)f^1_{w}\varepsilon^{m-n+1}_{s+1} + b_{m,r}(m-n+1,s)\varepsilon^{m-n+1}_{s} f^1_{w+1}\Big)  \\
& -  (-1)^{n-1}\psi_\eta( \sum_{j=0}^{t_{m-1}}<\varepsilon^{m-1}_j>_{m,r})\\
&= b_{m,r}(m-n+1,s+1)f^1_{w}d(\varepsilon^{m-n+1}_{s+1}) + b_{m,r}(m-n+1,s)d(\varepsilon^{m-n+1}_{s}) f^1_{w+1} \\
& - (-1)^{n-1} \sum_{j=0}^{t_{m-1}}\psi_\eta(<\varepsilon^{m-1}_j>_{m,r}).
\end{align*}
Since the map $\psi_\eta$ takes basis elements $\{\varepsilon^{m-1}_j\}_{j=0}^{t_{m-1}}$ of $\K_{m-1}$ to basis elements $\{\varepsilon^{m-n}_{j'}\}_{j'=0}^{t_{m-n}}$ of $\K_{m-n}$, the index $j$ shifts to $j'$ over the sum. Also apply the result of 
Lemma~\ref{contralem2} to obtain 
\begin{multline*}
 b_{m,r}(m-n+1,s+1)f^1_{w} \sum_{\alpha=0}^{t_{m-n}}<\varepsilon^{m-n}_\alpha>_{m-n+1,s+1}  \\
+  b_{m,r}(m-n+1,s) \sum_{\beta=0}^{t_{m-n}}<\varepsilon^{m-n}_\beta>_{m-n+1,s} f^1_{w+1}\\
 - (-1)^{n-1} \sum_{j'=0}^{t_{m-n}} b_{m-1,j}(m-n,j'+1)<f^1_w\varepsilon^{m-n}_{j'+1}>_{m-n+1,r} \\
+  b_{m-1,j}(m-n,j')<\varepsilon^{m-n}_{j'} f^1_{w+1}>_{m-n+1,r}
\end{multline*}
Applying the definition of $<\varepsilon^{*}_{*}>_{*,*}$ at the appropriate places, we obtain
\begin{multline*}
 b_{m,r}(m-n+1,s+1)f^1_{w} \sum_{\alpha=0}^{t_{m-n}}\Big[  \sum_{p=0}^{t_1}c_{p\alpha}(m-n+1,s+1,1)f_p^1 \varepsilon^{m-n}_{\alpha} \\
+ (-1)^{m-n+1} \sum_{q=0}^{t_1}c_{\alpha q}(m-n+1,s+1,m-n)\varepsilon^{m-n}_{\alpha} f_q^1\Big]\\
+  b_{m,r}(m-n+1,s) \sum_{\beta=0}^{t_{m-n}}\Big[ \sum_{p=0}^{t_1}c_{p\beta}(m-n+1,s,1)f_p^1 \varepsilon^{m-n}_{\beta} 
\\+ (-1)^{m-n+1} \sum_{q=0}^{t_1}c_{\beta q}(m-n+1,s,m-n)\varepsilon^{m-n}_{\beta} f_q^1 \Big] f^1_{w+1}
\end{multline*}
\begin{multline*}
 - (-1)^{n-1} \sum_{j'=0}^{t_{m-n}} b_{m-1,j}(m-n,j'+1) f^1_w \Big[  \sum_{p=0}^{t_1}c_{p,j'+1}(m-n+1,r,1)f_p^1 \varepsilon^{m-n}_{j'+1} 
\\+ (-1)^{m-n+1} \sum_{q=0}^{t_1}c_{j'+1, q}(m-n+1,r,m-n)\varepsilon^{m-n}_{j'+1} f_q^1\Big] \\
- (-1)^{n-1} \sum_{j'=0}^{t_{m-n}}  b_{m-1,j}(m-n,j')\Big[ \sum_{p=0}^{t_1}c_{pj'}(m-n+1,r,1)f_p^1 \varepsilon^{m-n}_{j'}
\\+ (-1)^{m-n+1} \sum_{q=0}^{t_1}c_{j' q}(m-n+1,r,m-n)\varepsilon^{m-n}_{j'} f_q^1 \Big] f^1_{w+1}.
\end{multline*}
After re-arranging, and bringing together like terms we get
\begin{align*}
&=\tag{a1}\sum_{\alpha=0}^{t_{m-n}} \sum_{p=0}^{t_1}b_{m,r}(m-n+1,s+1) c_{p\alpha}(m-n+1,s+1,1) f^1_{w}f_p^1 \varepsilon^{m-n}_{\alpha} \\
&\tag{a2} - (-1)^{n-1}\sum_{j'=0}^{t_{m-n}} \sum_{p=0}^{t_1}  b_{m-1,j}(m-n,j'+1)  c_{p,j'+1}(m-n+1,r,1) f^1_w f_p^1 \varepsilon^{m-n}_{j'+1}
\end{align*}
\begin{align*}
&\tag{b1}+(-1)^{m-n+1}\sum_{\alpha=0}^{t_{m-n}} \sum_{q=0}^{t_1} b_{m,r}(m-n+1,s+1) c_{\alpha q}(m-n+1,s+1,m-n) f^1_{w} \varepsilon^{m-n}_{\alpha} f_q^1  \\
&\tag{b2}- (-1)^{m} \sum_{j'=0}^{t_{m-n}} \sum_{q=0}^{t_1}  b_{m-1,j}(m-n,j'+1)  c_{j'+1, q}(m-n+1,r,m-n) f^1_w \varepsilon^{m-n}_{j'+1} f_q^1
\end{align*}
\begin{align*} 
&\tag{c1}+\sum_{\beta=0}^{t_{m-n}} \sum_{p=0}^{t_1} b_{m,r}(m-n+1,s) c_{p\beta}(m-n+1,s,1)f_p^1 \varepsilon^{m-n}_{\beta}f^1_{w+1} \\
&\tag{c2}- (-1)^{n-1} \sum_{j'=0}^{t_{m-n}} \sum_{p=0}^{t_1}b_{m-1,j}(m-n,j') c_{pj'}(m-n+1,r,1)f_p^1 \varepsilon^{m-n}_{j'}f^1_{w+1}
\end{align*}
\begin{align*}
&\tag{d1}+ (-1)^{m-n+1}\sum_{\beta=0}^{t_{m-n}} \sum_{q=0}^{t_1} b_{m,r}(m-n+1,s) c_{\beta q}(m-n+1,s,m-n)\varepsilon^{m-n}_{\beta} f_q^1 f^1_{w+1}  \\
&\tag{d2}(-1)^{m}\sum_{j'=0}^{t_{m-n}} \sum_{q=0}^{t_1} b_{m-1,j}(m-n,j') c_{j' q}(m-n+1,r,m-n)\varepsilon^{m-n}_{j'} f_q^1  f^1_{w+1} .
\end{align*}
Next we combine Equations (a1) and (a2) and re-index $\alpha=j'+1$, combine Equations (b1) and (b2) and re-index $\alpha=j'+1$ and so on to obtain 
\begin{align}\label{d-j-eq}
& \sum_{\alpha=0}^{t_{m-n}} \sum_{p=0}^{t_1}\Big[ A_{\alpha} \Big]f^1_{w}f_p^1 \varepsilon^{m-n}_{\alpha} + \sum_{\alpha=0}^{t_{m-n}} \sum_{q=0}^{t_1}\Big[ B_{\alpha} \Big]f^1_{w}\varepsilon^{m-n}_{\alpha}f^1_{q}  \notag
\\
&+\sum_{\beta=0}^{t_{m-n}} \sum_{p=0}^{t_1}\Big[ C_{\beta} \Big]f^1_p\varepsilon^{m-n}_{\beta}f^1_{w+1}
 + \sum_{\beta=0}^{t_{m-n}} \sum_{q=0}^{t_1}\Big[ D_{\beta} \Big]\varepsilon^{m-n}_{\beta}f^1_{q}f^1_{w+1}
\end{align}
where 
\begin{multline*}
A_{\alpha} =  b_{m,r}(m-n+1,s+1) c_{p\alpha}(m-n+1,s+1,1)\\
 + (-1)^{n} b_{m-1,j}(m-n,\alpha)  c_{p,\alpha}(m-n+1,r,1)\\
B_{\alpha} =   (-1)^{m-n+1} b_{m,r}(m-n+1,s+1) c_{\alpha q}(m-n+1,s+1,m-n)\\
 -(-1)^{m} b_{m-1,j}(m-n,\alpha) c_{\alpha,q}(m-n+1,r,m-n)\\
C_{\beta} =   b_{m,r}(m-n+1,s) c_{p \beta}(m-n+1,s,1)\qquad\qquad\qquad\\
 +(-1)^{n} b_{m-1,j}(m-n,\beta)  c_{p \alpha}(m-n+1,r,1)\\
D_{\beta} =   (-1)^{m-n+1} b_{m,r}(m-n+1,s) c_{\beta q}(m-n+1,s,m-n)\\
 +(-1)^{m}b_{m-1,j}(m-n,\beta)  c_{\beta,q}(m-n+1,r,m-n).
\end{multline*}

In the last expression that is  Equation \eqref{d-j-eq}, substitute
$$  \begin{cases}
 \sigma_{m,r}(i,j,r)  = A_{\alpha} &\;{\rm when}\;i=w\; (fixed), j=p, r=\alpha, 0\leq r\leq t_{m-n},\;0\leq j \leq t_1 \\
 \sigma_{m,r}(i,r,j)  = B_{\alpha} &\;{\rm when}\;i=w\;(fixed), j=q, r=\alpha, 0\leq r \leq t_{m-n},\;0\leq j \leq t_1 \\
 \sigma_{m,r}(i,r,j)  = C_{\beta} &\;{\rm when}\;\;i=p, j=w+1\;(fixed), r=\beta, 0\leq r\leq t_{m-n},\;0\leq i \leq t_1 \\
 \sigma_{m,r}(r,i,j)  = D_{\beta} &\;{\rm when}\;\;i=q, j=w+1\;(fixed), r=\beta, 0\leq r\leq t_{m-n},\;0\leq i \leq t_1 \\ \end{cases} $$
and $0$ otherwise. We obtain the desired result 
\begin{equation*}
(d\psi_\eta - (-1)^{n-1}\psi_\eta d )(\varepsilon^m_r) =  J(\varepsilon^m_r).
\end{equation*}
\end{proof}

\begin{theo}\label{homotopylifting2}
Let $\Lambda = kQ/I$ be a Koszul algebra. Suppose $\eta :\K_n\rightarrow \Lambda$ is a cocycle such that $\eta = \begin{pmatrix} 0 & \cdots & 0 & (f^1_wf^1_{w+1})^{(i)} & 0 & \cdots  & 0\end{pmatrix}$ for some $0\leq w\leq t_1$. A homotopy lifting map $\psi_\eta :\K_{m} \rightarrow \K_{m-n+1}$ associated to $\eta$ can be defined by 
$$\psi_\eta (\varepsilon^m_r) = b_{m,r}(m-n+1,s+1)f^1_{w}\varepsilon^{m-n+1}_{s+1} + b_{m,r}(m-n+1,s)\varepsilon^{m-n+1}_s f^1_{w+1}$$
for some scalars $b_{m,r}(m-n+1,s).$ Moreover, the scalars $b_{m,r}(m-n+1,s)$ satisfy
\begin{itemize}
\item[(i)]
\begin{multline*} A_{\alpha} = \begin{cases} c_{i,\alpha}(m,r,1), & {\rm if}\; p=w+1 \\ 0, &{\rm if}\;p\neq w+1, \end{cases},\\
D_{\beta} = \begin{cases}(-1)^{n(m-n)+1} c_{\beta,i}(m,r,m-n), & {\rm if}\; q=w \\ 0, &{\rm if}\;q\neq w \end{cases}
\end{multline*}
\item[(ii).] $B_{\alpha} =  0$ and  $C_{\beta} = 0$ for all $ \alpha$ and $\beta$ where 
\end{itemize}
\begin{multline*}
A_{\alpha} =  b_{m,r}(m-n+1,s+1) c_{p\alpha}(m-n+1,s+1,1)\\
 + (-1)^{n} b_{m-1,j}(m-n,\alpha)  c_{p,\alpha}(m-n+1,r,1),\\
B_{\alpha} =   (-1)^{m-n+1} b_{m,r}(m-n+1,s+1) c_{\alpha q}(m-n+1,s+1,m-n)\\
 -(-1)^{m} b_{m-1,j}(m-n,\alpha) c_{\alpha,q}(m-n+1,r,m-n),\\
C_{\beta} =   b_{m,r}(m-n+1,s) c_{p \beta}(m-n+1,s,1)\qquad\qquad\qquad\\
 +(-1)^{n} b_{m-1,j}(m-n,\beta)  c_{p \alpha}(m-n+1,r,1) \text{ and }\\
D_{\beta} =   (-1)^{m-n+1} b_{m,r}(m-n+1,s) c_{\beta q}(m-n+1,s,m-n)\\
 +(-1)^{m}b_{m-1,j}(m-n,\beta)  c_{\beta,q}(m-n+1,r,m-n).
\end{multline*}
\end{theo}
\begin{proof}
We have already established from Lemma~\ref{contralem3} that
$$(d\psi_\eta - (-1)^{n-1}\psi_\eta d )(\varepsilon^m_r) =  J(\varepsilon^m_r)$$
Applying definition (ii) of the theorem, that is substitute $B_\alpha=0$ and $C_\beta=0$ into Equation~\ref{d-j-eq} of Lemma~\ref{contralem3}. We get
\begin{align*}
&(d\psi_\eta - (-1)^{n-1}\psi_\eta d )(\varepsilon^m_r) \\
 &= \sum_{\alpha=0}^{t_{m-n}} \sum_{p=0}^{t_1}\Big[ A_{\alpha} \Big]f^1_{w}f_p^1 \varepsilon^{m-n}_{\alpha} + \sum_{\beta=0}^{t_{m-n}} \sum_{q=0}^{t_1}\Big[ D_{\beta} \Big]\varepsilon^{m-n}_{\beta}f^1_{q}f^1_{w+1}.
\end{align*}
After applying the definition of $A_\alpha$ and $D_\beta$ of (i) of the theorem into the above expression, we obtain
\begin{align*}
& \sum_{\alpha=0}^{t_{m-n}}c_{i,\alpha}(m,r,1) f^1_wf^1_{w+1}\varepsilon^{m-n}_{\alpha} -  (-1)^{n(m-n)}\sum_{\beta=0}^{t_{m-n}}c_{\beta,i}(m,r,m-n)\varepsilon^{m-n}_{\beta} f^1_w f^1_{w+1}.
\end{align*}

On the other hand, using the multiplicative structure on $\K$, we get \\
$\displaystyle{\Delta_{\K}(\varepsilon^m_r) = \sum_{v=0}^{m}\sum_{x=0}^{t_v}\sum_{y=0}^{t_{m-v}}c_{x,y}(m,r,v)\varepsilon^v_x\ot_{\Lambda}\varepsilon^{m-v}_y.}$ Applying $(\eta\ot 1 - 1\ot \eta )$, we obtain
\begin{multline*}
(\eta\ot 1 - 1\ot \eta )\Delta_{\K} (\varepsilon^m_r) = \sum_{v=0}^{m}\sum_{x=0}^{t_v}\sum_{y=0}^{t_{m-v}}c_{x,y}(m,r,v)(\eta\ot 1)(\varepsilon^v_x\ot_{\Lambda}\varepsilon^{m-v}_y)\\
 -\sum_{v=0}^{m}\sum_{x=0}^{t_v}\sum_{y=0}^{t_{m-v}}c_{x,y}(m,r,v) (1\ot \eta )(\varepsilon^v_x\ot_{\Lambda}\varepsilon^{m-v}_y)
\end{multline*}
whenever $v=n, x=i$ in the first part and $m-n=v, y=i$ in the second part, the above expression will yield

\begin{align*}
&= \sum_{y=0}^{t_{m-n}}c_{i,y}(m,r,1) (\eta\ot 1 )(\varepsilon^n_i\ot_{\Lambda}\varepsilon^{m-n}_y)  - \sum_{x=0}^{t_{m-n}}c_{x,i}(m,r,m-n) (1\ot \eta )(\varepsilon^{m-n}_x\ot_{\Lambda}\varepsilon^{n}_i)\\
&= \sum_{y=0}^{t_{m-n}}c_{i,y}(m,r,1)\eta(\varepsilon^n_i)\varepsilon^{m-n}_y -  (-1)^{n(m-n)}\sum_{x=0}^{t_{m-n}}c_{x,i}(m,r,m-n)\varepsilon^{m-n}_x\eta(\varepsilon^{n}_i)
\end{align*}
which after applying the definition of $\eta$ and re-indexing, we get
\begin{align*}
&= \sum_{y=0}^{t_{m-n}}c_{i,y}(m,r,1) f^1_wf^1_{w+1}\varepsilon^{m-n}_y -  (-1)^{n(m-n)}\sum_{x=0}^{t_{m-n}}c_{x,i}(m,r,m-n)\varepsilon^{m-n}_x f^1_w f^1_{w+1}
\end{align*}
\end{proof}

The following theorem gives a combinatorial description of what we obtain when the Gerstenhaber bracket of any two Hochschild cochains is applied to free basis elements.
\begin{theo}\label{Gerstenbrack}
Suppose that $\eta:\K_n\rightarrow\Lambda$ and $\theta:\K_m\rightarrow\Lambda$ represent elements in $\HH^*(\Lambda)$ and are given by  $ \eta = \begin{pmatrix} 0 & \cdots & 0 & (\lambda_i)^{(i)} & 0 & \cdots  & 0\end{pmatrix}$ and $ \theta = \begin{pmatrix} 0 & \cdots & 0 & (\lambda_j)^{(j)} & 0 & \cdots  & 0\end{pmatrix}$ for fixed $i,j$ where $0\leq i\leq t_n$ and $0\leq j\leq t_m.$ Then the bracket $[\eta, \theta] : \K_{n+m-1}\rightarrow\Lambda$ has the property that
\begin{align*}
[\eta, \theta](\varepsilon^{m+n-1}_r) \in &\begin{cases} kQ_1 &{\rm if}\;\lambda_i=f^1_i\;{\rm and }\;\lambda_j=f^1_j, \\
 kQ_2&{\rm if}\;\lambda_i=f^1_{i}f^1_{i+1}\;{\rm and}\;\lambda_j=f^1_j \\
kQ_3&{\rm if}\;\lambda_i=f^1_{i}f^1_{i+1}\;{\rm and }\;\lambda_j=f^1_j f^1_{j+1}.
\end{cases}
\end{align*}
\end{theo}
\begin{proof} Using the definition of Gerstenhaber bracket of Definition \ref{gbrac1}, we get
\begin{align*}
[\eta, \theta](\varepsilon^{m+n-1}_r) &= (\eta \psi_\theta - (-1)^{(m-1)(n-1)}\theta\psi_\eta)(\varepsilon^{m+n-1}_r)\\
&= \eta\psi_\theta(\varepsilon^{m+n-1}_r) -(-1)^{(m-1)(n-1)}\theta\psi_\eta (\varepsilon^{m+n-1}_r)
\end{align*}
We now apply the definition of a homotopy lifting map as given in Theorems~\ref{homotopylifting1} and~\ref{homotopylifting2}.\\
(1) Suppose that both are paths of length 1, i.e. $\lambda_i=f^1_w, \lambda_j=f^1_p$. We will get $\eta(b_{m+n-1,r}(n,r'')\cdot\varepsilon^{n}_{r''})  - (-1)^{(m-1)(n-1)} \theta (b_{m+n-1,r}(m,r')\cdot\varepsilon^{m}_{r'})$. This expression will give $0$ or a non-zero path. We are interested in the non-zero case i.e. when $r''=i$ (or $\eta(\varepsilon^n_{r''})=\lambda_{i}$) and $r'=j$ (or $\theta(\varepsilon^m_{r'})=\lambda_{j}$). This yields $b_{m+n-1,r}(n,i)\lambda_{i}  -  (-1)^{(m-1)(n-1)} b_{m+n-1,r}(m,j)\lambda_{j}$ which is a combination of paths of length 1.\\
(2) Suppose that one of them is a path of length 2, i.e.$\lambda_i=f^1_{w} f^1_{w+1},\;\lambda_j=f^1_p.$ The expression will yield $\eta(b_{m+n-1,r}(n,r'')\cdot\varepsilon^{n}_{r''})  - (-1)^{(m-1)(n-1)} \theta [ b_{m+n-1,r}(m,r'+1)f^1_{w}\cdot\varepsilon^{m}_{r'+1} + b_{m+n-1,r}(m,r')\cdot\varepsilon^{m}_{r'} f^1_{w+1}].$ Now consider cases in which we get non-zero paths:
$$\begin{cases}  
b_{m+n-1,r}(n,i)\lambda_{i}  -  (-1)^{(m-1)(n-1)} b_{m+n-1,r}(m,j)f^1_w\lambda_{j}  &{\rm if}\;r''=i, r'+1=j \\
 b_{m+n-1,r}(n,i)\lambda_{i}  -  (-1)^{(m-1)(n-1)} b_{m+n-1,r}(m,j)\lambda_{j}f^1_{w+1} &{\rm if}\;r''=i, r'=j
\end{cases}$$
which are combination of paths of lenth 2.\\
(3) Suppose that both are paths of length 2. Let $\lambda_i=f^1_{w}f^1_{w+1} ,\;\lambda_j=f^1_{p} f^1_{p+1},$ we obtain for the bracket expression
\begin{multline*}
\eta[ b_{m+n-1,r}(n,r''+1)f^1_{p}\cdot\varepsilon^{n}_{r''+1} + b_{m+n-1,r}(n,r'')\cdot\varepsilon^{n}_{r''} f^1_{p+1}]  \\
- (-1)^{(m-1)(n-1)} \theta [ b_{m+n-1,r}(m,r'+1)f^1_{w}\cdot\varepsilon^{m}_{r'+1} + b_{m+n-1,r}(m,r')\cdot\varepsilon^{m}_{r'} f^1_{w+1} ].
\end{multline*}
After eliminating the cases which becomes 0, what we get is 
$$\begin{cases} 
 b_{m+n-1,r}(n,i)f^1_p\lambda_{i}  -  (-1)^{(m-1)(n-1)} b_{m+n-1,r}(m,j)f^1_w \lambda_{j} &{\rm if}\;r''+1=i, r'+1=j\\
 b_{m+n-1,r}(n,i)f^1_p\lambda_{i}  -  (-1)^{(m-1)(n-1)} b_{m+n-1,r}(m,j)\lambda_{j}f^1_{w+1} &{\rm if}\;r''+1=i, r'=j\\
  b_{m+n-1,r}(n,i)\lambda_{i} f^1_p  -  (-1)^{(m-1)(n-1)} b_{m+n-1,r}(m,j) f^1_{w}\lambda_{j}&{\rm if}\;r''=i, r'+1=j\\
b_{m+n-1,r}(n,i)\lambda_{i}f^1_{p+1}  -  (-1)^{(m-1)(n-1)} b_{m+n-1,r}(m,j)\lambda_{j} f^1_w &{\rm if}\;r''=i, r'=j
\end{cases}$$
which are all linear combination of paths of length 3.
\end{proof}

Any cocycle $\eta$ of degree $n$ can be thought of as a sum of maps $\eta = \sum_{i=0}^{t_n} \eta^{(i)}$ where $ \eta^{(i)}$ is the map taking the $i$-th basis element $\varepsilon^n_i$ to $\lambda,$ a non-zero element of the algebra, and all other basis elements $\varepsilon^n_s$ to $0$, $i\neq s$. Consistent with our notation, this map is written $\eta^{(i)}= \begin{pmatrix} 0 & \cdots & 0 & (\lambda)^{(i)} & 0 & \cdots  & 0\end{pmatrix},$ and we use this notation in the following theorem.

\begin{theo}\label{Gerstenbrack1}
Let $\eta:\K_n\rightarrow\Lambda$ and $\theta:\K_m\rightarrow\Lambda$ represent elements in $\HH^*(\Lambda)$ and are given by $\eta(\varepsilon^n_i) = \lambda_i$ for $i=0,1,\ldots,t_n$ and $\theta(\varepsilon^m_j) = \beta_j$ for $j=0,1,\ldots,t_m.$ Then the $r$-component of the bracket $[\eta, \theta] : \K_{n+m-1}\rightarrow\Lambda$ denoted by $[\eta, \theta]^{(r)}$ can be expressed on the $r$-th basis element $\varepsilon^{m+n-1}_r$ as
\begin{align*}
&[\eta, \theta]^{(r)}(\varepsilon^{m+n-1}_r) = \sum_{i=0}^{t_n} \sum_{j=0}^{t_m}  b_{m-n+1,r}(n,i)\lambda_i   - (-1)^{(m-1)(n-1)}(b_{m-n+1,r}(m,j)\beta_j
\end{align*}
provided $\lambda_i,\beta_j$ are all paths of length 1.
\end{theo}
\begin{proof} 
Write $\eta = \sum_{i=0}^{t_n} \eta^{(i)}$ where $ \eta^{(i)} = \begin{pmatrix} 0 & \cdots & 0 & (\lambda_i)^{(i)} & 0 & \cdots  & 0\end{pmatrix}$ and $\theta = \sum_{j=0}^{t_m} \theta^{(j)}$ where $ \theta^{(j)} = \begin{pmatrix} 0 & \cdots & 0 & (\beta_j)^{(j)} & 0 & \cdots  & 0\end{pmatrix}.$ Using the definition of Gerstenhaber bracket of~\ref{gbrac1}, we get for $0\leq r\leq t_{m+n-1}$,
\begin{align*}
&[\eta, \theta]^{(r)}(\varepsilon^{m+n-1}_r) = [\sum_{i=0}^{t_n} \eta^{(i)},  \sum_{j=0}^{t_m} \theta^{(j)}]^{r}(\varepsilon^{m+n-1}_r) = \sum_{i=0}^{t_n} \sum_{j=0}^{t_m} [\eta^{(i)},  \theta^{(j)}]^{r}(\varepsilon^{m+n-1}_r)\\
&=\sum_{i=0}^{t_n} \sum_{j=0}^{t_m} (\eta^{(i)} \psi_{\theta^{(j)}} - (-1)^{(m-1)(n-1)}\theta^{(j)}\psi_{\eta^{(i)}})(\varepsilon^{m+n-1}_r)
\end{align*}
Since $\lambda_i = f^1_{w_i}$ for all $i$, and $\beta_j = f^1_{p_j}$ for all $j$, then the homotopy lifting maps can be defined as
$\psi_{\theta^{(j)}}(\varepsilon^{m+n-1}_r) = b_{m-n+1,r}(n,i)\varepsilon^n_i$ for some $i$ and $\psi_{\eta^{(i)}}(\varepsilon^{m+n-1}_r) = b_{m-n+1,r}(m,j)\varepsilon^m_j$ for some $j$. Applying this, we get
\begin{align*}
&\sum_{i=0}^{t_n} \sum_{j=0}^{t_m} \eta^{(i)}( b_{m-n+1,r}(n,i)\varepsilon^n_i  )  - (-1)^{(m-1)(n-1)}\theta^{(j)}(b_{m-n+1,r}(m,j)\varepsilon^m_j)\\
&=\sum_{i=0}^{t_n} \sum_{j=0}^{t_m}  b_{m-n+1,r}(n,i)\lambda_i   - (-1)^{(m-1)(n-1)}(b_{m-n+1,r}(m,j)\beta_j
\end{align*}
\end{proof}

\section{Application}\label{Gerstenhaber-app}
In this section, we give an application of our results in Theorems \ref{homotopylifting1} and \ref{homotopylifting2} to specify solutions to the Maurer-Cartan equation. The Maurer-Cartan equation is useful in the theory of deformation of algebras.

It is known that Hochschild cohomology is a differential graded Lie algebra i.e. 
 $$\bar{d}([f,g]) = [\bar{d}(f),g] + (-1)^{(m-1)}[f,\bar{d}(g)],$$
 where $\bar{d}(f) = (-1)^{(m-1)}f\delta_{m+1}$. Since the resolution $(\K,d)$ embeds into the bar reduced resolution $(\mathcal{B},\delta)$ via $\K\xrightarrow{\iota}\mathcal{B}$, with $ \iota d = \delta\iota$, there are no sign changes, hence we take $\bar{d}(\eta) = (-1)^{(m-1)}d^{*}_{m+1}\eta = (-1)^{(m-1)}\eta d_{m+1}.$ An Hochschild 2-cocycle $\eta$ is then said to satisfy the Maurer-Cartan equation if
\begin{equation}\label{maurer-cartan}
\bar{d}(\eta) + \frac{1}{2}[\eta,\eta] = 0 
\end{equation}
Applying the definition of the bracket to a 2-cocycle, we obtain the following version of the Maurer-Cartan equation
\begin{equation}\label{maurer-cartan1}
 - d^{*}_{3}(\eta) = - \frac{1}{2}(\eta\psi_\eta + \eta\psi_\eta) = - \eta\psi_\eta
\end{equation}
\begin{theo}\label{Maurer-cartan-theorem}
Let $k$ be a field and $\Lambda = kQ/I$ be a Koszul algebra. Suppose $\eta :\K_2\rightarrow \Lambda$ is a cocycle such that 
$ \eta = \begin{pmatrix} 0 & \cdots & 0 & (\lambda)^{(i)} & 0 & \cdots  & 0\end{pmatrix}$, $i=0,1,2,\ldots, t_2$. If $\lambda\in kQ_2$, i.e. a linear combination of paths of length 2, then $\eta$ satisfies the Maurer-Cartan equation.
\end{theo}
\begin{proof}
It is enough to check that for the case where $\lambda = f^1_u f^1_v,$ the result is true. The left hand side of Equation~\ref{maurer-cartan1} is given by
\begin{align*}
 & d^{*}_{3}(\eta)(\varepsilon^3_r) = \eta d_3 (\varepsilon^3_r) = \eta \Big[ \sum_{j=0}^{t_2}<\varepsilon^{2}_j>_{3,r} \Big]  = \eta\sum_{j=0}^{t_2}\Big[ \sum_{p=0}^{t_1}c_{pj}(3,r,1)f_p^1 \varepsilon^{2}_j 
- \sum_{q=0}^{t_1}c_{jq}(3,r,2)\varepsilon^{2}_j f_q^1 \Big] \\
&= \sum_{p=0}^{t_1}c_{pi}(3,r,1)f_p^1 \eta(\varepsilon^{2}_i) 
- \sum_{q=0}^{t_1}c_{iq}(3,r,2)\eta(\varepsilon^{2}_i) f_q^1   \\
&= \sum_{p=0}^{t_1}c_{pi}(3,r,1)f_p^1  f^1_u f^1_v 
- \sum_{q=0}^{t_1}c_{iq}(3,r,2) f^1_u f^1_v f_q^1 
\end{align*}
which is in $kQ_3$. On the other hand, the result of Theorem \ref{Gerstenbrack} states that $[\eta, \eta](\varepsilon^{3}_r) \in kQ_3$ since $\lambda= f^1_u f^1_v,$ a path of length 2. In particular, the right hand side of Equation \eqref{maurer-cartan1} becomes
\begin{align*}
&\frac{1}{2}[\eta, \eta](\varepsilon^{3}_r)  = \eta\psi_\eta(\varepsilon^{3}_r) = \eta( b_{3,r}(2,i)f^1_{u}\varepsilon^{2}_{i} + b_{3,r}(2,j)\varepsilon^{2}_j f^1_{v})\\ 
&= b_{3,r}(2,i)f^1_{u}\eta(\varepsilon^{2}_{i}) =  b_{3,r}(2,i) f^1_{u} f^1_u f^1_v.
\end{align*}
This expression could have yielded $ b_{3,r}(2,i) f^1_{u} f^1_v f^1_v$ if we take $\psi_\eta(\varepsilon^3_r)$ to be $b_{3,r}(2,j)f^1_{u}\varepsilon^{2}_{j} + b_{3,r}(2,i)\varepsilon^{2}_i f^1_{v}$ but it doesn't affect the result. The solution to this equation is obtained by identifying possible scalars for which
\begin{align*}
& \sum_{p=0}^{t_1}c_{pi}(3,r,1)f_p^1  f^1_u f^1_v 
- \sum_{q=0}^{t_1}c_{iq}(3,r,2) f^1_u f^1_v f_q^1  =   b_{3,r}(2,i) f^1_{u} f^1_u f^1_v.
\end{align*}
If $c_{pi}(3,r,1) = 0$ for all $p\neq u$,  $c_{iq}(3,r,2) = 0$ for all $q$ and $c_{pi}(3,r,1) = b_{3,r}(2,i)$ when $p=u$, the Maurer-Cartan equation \eqref{maurer-cartan} holds. 
\begin{rema}
A 2-cocycle defined with $\lambda =f^1_u$ cannot satisfy the Maurer-Cartan equation. This is because while the left hand side of Equation~\ref{maurer-cartan1} yields a linear combination of paths of length 2, the right hand side yields a linear combination of paths of length 1.
\end{rema}

\end{proof}

\section{Short Example}\label{sexample}
Let $k$ be a field of characteristics different from 2. Consider the quiver algebra $A=kQ/I$ (also examined in \cite[ Example 5]{MSKA}) defined using the following finite quiver:
$$
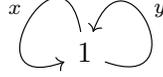
\begin{tikzcd} 1 \arrow[out=100,in=200,loop,swap,"x"]
  \arrow[out=340,in=70,loop,swap,"y"]
\end{tikzcd}$$
with one vertex and two arrows  $x,y.$ We denote by $e_1$ the idempotent associated with the only vertex. Let $I$, an ideal of the path algebra $kQ$ be defined by
\begin{equation*}
I = \langle x^2,xy+yx\rangle.
\end{equation*}
Since $\{x^2, xy+yx\}$ is a quadratic Grobner basis for the ideal generated by relations under the length lexicographich order with $x>y>1$, the algebra is Koszul.

In order to define a comultiplicative structure, we take $t_0=0, t_n=1$ for all $n$, $f_0^0=e_1, f_1^0=0, f_0^1=x, f_1^1=y, f^2_0=x^2, f^2_1=xy+yx, f_0^3=x^3, f_1^3 = x^2y+xyx+yx^2,$ and in general $ f^n_0 = x^n, f^n_1 = \sum_{i+j=n-1}x^iyx^j$. We also see that $f^n_0 = f^r_0 f^{n-r}_0$ and $f^n_1 = f^r_0f^{n-r}_1 + f^r_1f^{n-r}_0$ so $c_{00}(n,0,r) = c_{01}(n,1,r) = c_{10}(n,1,r) = 1$ and all other $c_{pq}(n,i,r)=0$. With the above stated, we can construct the resolution $\K$ for the algebra $A$. A calculation shows that 
\begin{align*}
&d_1(\varepsilon^1_0) = x\varepsilon^0_0 - \varepsilon^0_0 x, && d_1(\varepsilon^1_1) = y\varepsilon^0_0 - \varepsilon^0_0 y\\
&d_2(\varepsilon^2_0) = x\varepsilon^1_0 + \varepsilon^1_0 x, && d_2(\varepsilon^2_1)= y\varepsilon^1_0 + \varepsilon^1_0 y + x\varepsilon^1_1 + \varepsilon^1_1 x.
\end{align*}
Consider the following maps $\chi,\theta:\K_1\xrightarrow{}A$ defined by $\chi = (xy\;\; 0)$ and $\theta = (0\;\; y)$. Observe that $\chi + \theta = (xy\;\; y)$ was originally given as an example in \cite{MSKA}. With the following calculations
\begin{align*}
\chi d_2(\varepsilon^2_0) &= \chi (x\varepsilon^1_0 + \varepsilon^1_0 x) = x^2y + xyx = x^2y - x^2y = 0 \\
\chi d_2(\varepsilon^2_1) &= \chi (y\varepsilon^1_0 + \varepsilon^1_0 y + x\varepsilon^1_1 + \varepsilon^1_1 x) = yxy + xy^2 + 0 = (yx + xy)y = 0 \\
\theta d_2(\varepsilon^2_0) &= \theta (x\varepsilon^1_0 + \varepsilon^1_0 x) = 0 \\
\theta d_2(\varepsilon^2_1) &= \theta (y\varepsilon^1_0 + \varepsilon^1_0 y + x\varepsilon^1_1 + \varepsilon^1_1 x) = 0 + xy + yx = 0,
\end{align*}
$\chi$ and $\theta$ are cocycles. The comultiplicative map $\Delta:\K\xrightarrow{}\K\ot_A\K$ on $\varepsilon^1_0$, $\varepsilon^1_1, \varepsilon^2_0,\varepsilon^2_1$ is given by 
\begin{align*}
\Delta(\varepsilon^1_0) &= c_{00}(1,0,0)\varepsilon^0_0\ot\varepsilon^1_0 + c_{00}(1,0,1)\varepsilon^1_0\ot\varepsilon^0_0=  \varepsilon^0_0\ot\varepsilon^1_0 + \varepsilon^1_0\ot\varepsilon^0_0,\\
\Delta(\varepsilon^1_1) &= \varepsilon^0_0\ot\varepsilon^1_1 + \varepsilon^1_1\ot\varepsilon^0_0,\\
\Delta(\varepsilon^2_0) &= \varepsilon^0_0\ot\varepsilon^2_0 + \varepsilon^1_0\ot\varepsilon^1_0 +\varepsilon^2_0\ot\varepsilon^0_0,\\
\Delta(\varepsilon^2_1) &=  \varepsilon^0_0\ot\varepsilon^2_1 + \varepsilon^1_0\ot\varepsilon^1_1 +\varepsilon^1_1\ot\varepsilon^1_0 + \varepsilon^2_1\ot\varepsilon^0_0.
\end{align*}
From Theorems \ref{homotopylifting1} and \ref{homotopylifting2}, it can be verified by direct calculations that the first, second and third degree of the homotopy lifting maps $\psi_\chi$ and $\psi_\theta$ associated with $\chi$ and $\eta$ respectively are the following:
\begin{align*}
\psi_{\chi_0} = 0, \qquad & \psi_{\chi_1}(\varepsilon^1_0) = x\varepsilon^1_1 + \varepsilon^1_0y, \psi_{\chi_1}(\varepsilon^1_1) = 0 && \psi_{\chi_2}(\varepsilon^2_0) = x\varepsilon^2_1, \psi_{\chi_2}(\varepsilon^2_1) = \varepsilon^2_1 y \\
\psi_{\theta_0} = 0,\qquad & \psi_{\theta_1}(\varepsilon^1_0) = 0, \qquad \psi_{\theta_1}(\varepsilon^1_1) = \varepsilon^1_1 && \psi_{\theta_2}(\varepsilon^2_0) = 0, \psi_{\theta_2}(\varepsilon^2_1) = \varepsilon^2_1. 
\end{align*}
For the cocycle $\chi$ for instance, $b_{20}(2,1)=b_{2,1}(2,1)=1$, $b_{10}(1,0) = b_{10}(1,1)=1$ and the Gerstenhaber bracket of $\chi$ and $\theta$ is given by $[\chi,\theta] = -\chi.$


\section{Long Example}\label{wexamples}
In this section, we give four examples of homotopy lifting maps that come from a family of quiver algebras. While a member of this family was introduced in~\cite{SVHC} as a counterexample to the Snashall-Solberg finite generation conjecture, the Hochschild cohomology modulo nilpotent cocycles of this family as a whole was studied in~\cite{TNO}. See~\cite{SVHC} for details on the Snashall-Solberg finite generation conjecture. Two of these homotopy lifting maps come from degree 1 cocycles while the other two are homotopy lifting maps for degree 2 cocycles. We adapted a version of the resolution $\K$ for this family and find cocycles. We especially note how these homotopy lifting maps follow from their generalized versions presented in Theorem~\ref{homotopylifting1} and Theorem~\ref{homotopylifting2} of Section~\ref{major}. We round off the section by making connections to derivation operators.\\

We begin with the following finite quiver:
$$Q:=  \begin{tikzcd} 1 \arrow[out=190,in=270,loop,swap,"b"]
  \arrow[out=90,in=170,loop,swap,"a"]
  \arrow[r,"c"] & 2
\end{tikzcd}$$
with two vertices and three arrows  $a,b,c.$ We denote by $e_1$ and $e_2$ the idempotents associated with vertices 1 and 2. Let $kQ$ be the path algebra associated with $Q$ and take for each $q\in k$, $I_q\subseteq kQ$ to be an admissible ideal of $kQ$ generated as follows
\begin{equation*}
I_q = \langle a^2,b^2,ab-qba, ac\rangle.
\end{equation*}
Let $\{ \Lambda_q\}_{q\in k} =\{ kQ/I_q \}_{q\in k} $ be a family of quiver algebras generated by the quiver $Q$ and the ideal $I_q$.
To define a set of free basis for the resolution $\K$ we start by letting $kQ_0$ to be the ideal of $kQ$ generated by the vertices of $Q$ with basis $f_0^0=e_1,f_1^0=e_2.$ Next, set $kQ_1$ to be the ideal generated by paths with basis $f_0^1=a,f_1^1=b$ and $f_2^1=c.$ Set $f^2_j,$ $j=0,1,2,3$ to be the set of paths of length 2 that generates the ideal $I$, that is
$f_0^2=a^2,f_1^2=ab-qba,f_2^2=b^2, f^2_3= ac,$ and define a comultiplicative equation on the paths of length $n>2$ in the following way.
$$\begin{cases} 
f^n_0 = a^{n}, \\
f^n_s = f^{n-1}_{s-1} b + (-q)^s f^{n-1}_s a ,& (0<s<n), \\
f^n_n = b^{n}, \\
f^n_{n+1} = a^{(n-1)} c,
\end{cases}$$
The resolution $\K \rightarrow\Lambda_q$ has free basis elements $\{\varepsilon^n_i\}_{i=0}^{t_n}$ such that for each $i$, we have $\varepsilon^n_i = (0,\ldots,0,o(f^n_i)\otimes_k t(f^n_i),0,\ldots,0)$. The differentials on $\K_n$ are given explicitly for this family  by
\begin{align*}
d_1(\varepsilon^1_2) &= c\varepsilon^0_1 - \varepsilon^0_0 c\\
d_n(\varepsilon^n_r) &= (1-\partial_{n,r})[a\varepsilon^{n-1}_r)+(-1)^{n-r}q^r\varepsilon^{n-1}_r a] \\
&+(1-\partial_{r,0})[(-q)^{n-r}b\varepsilon^{n-1}_{r-1}+(-1)^{n}\varepsilon^{n-1}_{r-1} b], \;\;\text{for}\;\;r\leq n\\
d_n(\varepsilon^n_{n+1}) &= a\varepsilon^{n-1}_n + (-1)^n\varepsilon^{n-1}_0 c, \;\;\text{when}\;\;n\geq 2,
\end{align*}
where $\partial_{r,s} = 1$ when $r=s$ and $0$ when $r\neq s$.

For any member $\Lambda_q$ of the above family of quiver algebras, we define the bar resolution $(\mathcal{B}, \delta)$ of $\Lambda_q$ by $\mathcal{B}_n = (\Lambda_q)^{\otimes_{\Lambda^{*}_0}(n+2)}$,  the $(n+ 2)$-fold tensor product of $\Lambda_q$ over $\Lambda^{*}_0$. Note here that $\Lambda^{*}_0$ comes from the grading on $\Lambda_q$ for each $q$ and is not the same as the algebra $\Lambda_q,q=0$. It is isomorphic to two copies of $k$ generated by the two vertices of the quiver $Q$.
Before we begin to present the homotopy lifting maps, we give explicit description of the comultiplicative map $\Delta_{\K}:\K\rightarrow \K\ot_{\Lambda_q}\K$ on the resolution $\K$. As was mentioned in the preliminaries, the resolution $\K$ embeds into the reduced bar resolution $\mathcal{B}$ via the map $\iota:\K_n\rightarrow\mathcal{B}_n$ defined by $\varepsilon^n_r\mapsto 1\ot \widetilde{f^n_r}\ot 1$, where each $\widetilde{f^n_r}$ is viewed as a sum of tensor product of paths of length 1 tensored over $\Lambda^{*}_0$ as given in Equation~\eqref{the-f}. For example, $\widetilde{f^2_0} = f^1_0\ot f^1_0 = a\ot a$, $\widetilde{f^2_1}= f^1_0\ot f^1_1 - q f^1_1\ot f^1_0 = a\ot b - q b\ot a.$ See~\cite[Proposition 2.1]{MSKA} for details about this. Equation~\eqref{diag-bar} gives the comultiplicative map $\Delta:\mathcal{B}\rightarrow \mathcal{B}\ot_{\Lambda_q}\mathcal{B}$ on the bar resolution making the following diagram commutative. 
\begin{equation*}
\begin{tikzcd}
\K \arrow{r}{\Delta_{\K}} \arrow[swap]{d}{\iota} & \K\ot_{\Lambda_q}\K  \arrow{d}{\iota\ot\iota} \\%
\mathcal{B} \arrow{r}{\Delta}& \mathcal{B}\ot_{\Lambda_q}\mathcal{B}.
\end{tikzcd}
\end{equation*}
It is clear that the following holds;
\begin{equation}\label{comultiplication-two}
\begin{cases} 
\widetilde{f^n_0} = f^{1}_0\ot f^{1}_0\ot \cdots \ot f^{1}_0, &{\rm n-times} , \\
\widetilde{f^n_s} = \widetilde{f^{n-1}_{s-1}}\ot f^1_1 + (-q)^s \widetilde{f^{n-1}_s}\ot f^1_0,& {\rm (0<s<n)}, \\
\widetilde{f^n_n} = f^{1}_1\ot f^{1}_1\ot \cdots \ot f^{1}_1, & {\rm n-times}, \\
\widetilde{f^n_{n+1}} = f^{1}_0\ot f^{1}_0\ot \cdots \ot f^{1}_0\ot f^1_2, & {\rm f^1_0 \text{ appearing } (n-1)-times},
\end{cases}
\end{equation}
In case $0<s<n$, it was also shown in \cite{TNO} that\\
 $f^n_s = \sum_{j=max\{0,r+t-n\}}^{min\{t,s\}} (-q)^{j(n-s+j-t)} f^t_j f^{n-t}_{s-j},$
hence, 
\begin{equation}
\widetilde{f^n_s} = \sum_{j=max\{0,r+t-n\}}^{min\{t,s\}} (-q)^{j(n-s+j-t)} \widetilde{f^t_j}\ot \widetilde{f^{n-t}_{s-j}}.
\end{equation}

Calculations from~\cite{TNO} show that for this family, the comultiplicative map can be expressed in the following way
\begin{equation*}
\Delta_{\K}(\varepsilon^{n}_{s}) =
\begin{cases}
\displaystyle{\sum_{r=0}^{n} \varepsilon^r_0\ot \varepsilon^{n-r}_0,} &  s=0 \\
\displaystyle{ \sum_{w=0}^{n}\sum_{j=max\{0,s+w-n\}}^{min\{w,s\}} (-q)^{j(n-s+j-w)} \varepsilon^w_j\ot \varepsilon^{n-w}_{s-j}, } & 0< s < n\\
\displaystyle{\sum_{t=0}^{n} \varepsilon^t_t\ot \varepsilon^{n-t}_{n-t}, } & s=n \\
\displaystyle{ \varepsilon^0_0\ot \varepsilon^{n}_{n+1} + \Big[ \sum_{t=0}^{n} \varepsilon^t_0\ot \varepsilon^{n-t}_{n-t+1}\Big] +  \varepsilon^n_{n+1}\ot \varepsilon^{0}_{0},} &s=n+1.
\end{cases} 
\end{equation*}

\subsection{Homotopy lifting maps from degree 2 cocycles}\label{deg2cocycles}
In this section, we will find Hochschild 2 cocycles. From these cocycles we will select two cocycles and find homotopy lifting maps associated to them. Our choice was arbitrary. However, these examples demonstrate that any homotopy lifting map arises in the form  described in Theorems~\ref{homotopylifting1} and \ref{homotopylifting2}.
\\
\textbf{Notation:} As it was previously used in Section \ref{major}, we recall and use the following standard notation for ease of computations. Since the set $\{\varepsilon^n_r\}_{r=0}^{t_n},$ $t_n=n+1$ forms a basis for $\K_n$, any module homomorphism $\Theta:\K_n\rightarrow\Lambda_q$ taking $\varepsilon^n_i$ to $\lambda_i$, $i=0,1,\ldots,n+1$ is denoted by
$$ \Theta = \begin{pmatrix} \lambda_{0} & \lambda_{1} & \cdots  &  \lambda_{t_n}\end{pmatrix}.$$

Suppose that the $\Lambda_q^e$-module homomorphism $\eta:\K_2\rightarrow \Lambda_q$ defined by \\
$ \eta
 = \begin{pmatrix} \lambda_{0} & \lambda_{1}& \lambda_{2} &  \lambda_{3}\end{pmatrix}$  is a cocycle, that is $d^*\eta = 0,$ with $\lambda_i\in\Lambda_q$ for all $i$. Since $d^*\eta : \K_3\rightarrow \Lambda_q$, we obtain using $d^*\eta(\varepsilon^3_r) = \eta d(\varepsilon^3_r)$, 
$$
\eta\cdot d ( \varepsilon^{3}_{i}) = 
\eta \Big(\begin{cases} a\varepsilon^{2}_{0} - \varepsilon^{2}_{0}a & {\rm if}\;i=0\\
 a\varepsilon^{2}_{1}+q\varepsilon^{2}_{1}a + q^2 b\varepsilon^{2}_{0} - \varepsilon^{2}_{0}b & {\rm if}\;i=1\\ 
 a\varepsilon^{2}_{2}-q^2\varepsilon^{2}_{2}a - q b\varepsilon^{2}_{1} - \varepsilon^{2}_{1}b & {\rm if}\;i=2\\ 
 b\varepsilon^{2}_{2}- \varepsilon^{2}_{2}b   & {\rm if}\;i=3\\   
 a\varepsilon^{2}_{3}- \varepsilon^{2}_{0}c  & {\rm if}\;i=4
 \end{cases}\Big)  
$$
$\eta\cdot d$ may then be identified with the matrix
$$
 \begin{pmatrix} a\lambda_{0} - \lambda_{0}a,  &
 a\lambda_{1}+q\lambda_{1}a + q^2 b\lambda_{0} - \lambda_{0}b, &
 a\lambda_{2}-q^2\lambda_{2}a - q b\lambda_{1} - \lambda_{1}b, &
 b\lambda_{2}- \lambda_{2}b,  &
 a\lambda_{3}- \lambda_{0}c 
 \end{pmatrix} $$
which will be equated to $\begin{pmatrix} 0 &0&0&0&0\end{pmatrix}$ and solved. We solve this system of equations with the following in mind. There is an isomorphism of $\Lambda_q^e$-modules $\HHom_{\Lambda^e}(\Lambda o(f^n_i)\otimes_k t(f^n_i)\Lambda,\Lambda) \simeq o(f^n_i)\Lambda \; t(f^n_i)$ ensuring that
\begin{align*}
o(f^2_i)\lambda_i t(f^2_i) &= o(f^2_i)\eta(\varepsilon^2_i) t(f^2_i) = o(f^2_i)\eta(o(f^2_i)\otimes_k t(f^2_i)) t(f^2_i)\\
 &= \phi(o(f^2_i)^2\otimes_k t(f^2_i)^2) = \phi(o(f^2_i)\otimes_k t(f^2_i)) = \lambda_i.
 \end{align*}
This means that for $i=0,1,2$ each $\lambda_i$ should satisfy $e_1\lambda_ie_1 = \lambda_i$  since the origin and terminal vertex of $f^2_0, f^2_1, f^2_2$ is $e_1$ and  $e_1\lambda_3e_2 = \lambda_3$. We obtain the following 9 solutions presented in Table \ref{table 1}.
\begin{table}[t]
\begin{tabular}{ |c|c|c|c|c|c|c|c|c|c| }  \hline
solutions             & 1 & 2  & 3 & 4 & 5 & 6 & 7 & 8 & 9 \\ \hline
$\lambda_0$ & a & ab & 0 & 0 & 0 & 0 & 0 & 0 & 0 \\ 
$\lambda_1$ & 0 & 0   & 0 & 0 & 0 & 0 & ab & 0 & 0 \\
$\lambda_2$ & 0 & 0   & a & b & ab& $e_1$ & 0 & 0 & 0 \\ 
$\lambda_3$ & 0 & 0   & 0 & 0 & 0 & 0 & 0 & c & bc \\  \hline
\end{tabular}
\caption{Possible values of $\eta(\varepsilon^2_r) = \lambda_r$ for different $r$.}
\label{table 1}
\end{table}
For the rest of this section, we are interested in the first and fifth cocycles, that is
$$ \bar{\eta} = \begin{pmatrix} a & 0 & 0 &  0 \end{pmatrix} \text{  and  }  \bar{\chi} = \begin{pmatrix} 0 & 0 & ab &  0 \end{pmatrix} .$$
Again we are looking at the homotopy lifting maps corresponding to these two cocycles because all of the other homotopy lifting maps corresponding to other cocyles will have similar properties. 
\\
\textbf{Homotopy lifting} for the first and fifth maps $\bar{\eta}$ and $\bar{\chi}$ will be maps $\psi_{\bar{\eta}}, \psi_{\bar{\chi}}:\K\rightarrow \K[-1]$ such that 
$$d\psi_{\bar{\eta}} +  \psi_{\bar{\eta}} d = (\bar{\eta}\ot 1_{\K} - 1_{\K}\ot \bar{\eta})\Delta_{\K},\;\text{ and }\; d\psi_{\bar{\chi}} +  \psi_{\bar{\chi}} d = (\bar{\chi}\ot 1_{\K} - 1_{\K}\ot \bar{\chi})\Delta_{\K}.$$ We track the left hand side of the above equations using the following diagram, which is not necessarily commutative.
$$ \begin{tikzcd}
\K: =\arrow{d}{\psi_{\bar{\eta}},\;\psi_{\bar{\chi}}} \cdots \arrow{r}{d_5} 
    & \K_4\arrow{d}{\psi_{\bar{\eta}_4},\;\psi_{\bar{\chi}_4}}\arrow{r}{d_4}
        & \K_3 \arrow{d}{\psi_{\bar{\eta}_3},\;\psi_{\bar{\chi}_3}} \arrow{r}{d_3}
            & \K_2 \arrow{d}{\psi_{\bar{\eta}_2},\;\psi_{\bar{\chi}_2}} \arrow{r}{d_2}
                & \K_1 \arrow{d}{\psi_{\bar{\eta}_1},\;\psi_{\bar{\chi}_1}} \arrow{r}{d_1}
		& \K_0\\
\K:= \cdots \arrow{r}{d_4}
    & \K_3 \arrow{r}{d_3}
        & \K_2 \arrow{r}{d_2}
            & \K_1 \arrow{r}{d_1}
                & \K_0
\end{tikzcd}$$
We now define $\psi_{\bar{\eta}_1}, \psi_{\bar{\eta}_2}, \psi_{\bar{\eta}_3}$ and  $\psi_{\bar{\chi}_1}, \psi_{\bar{\chi}_2}, \psi_{\bar{\chi}_3}$ only, just to point out that homotopy lifting maps can be defined as generalized in Theorem~\eqref{homotopylifting1} and~\eqref{homotopylifting1}. Calculations show that
\begin{equation}\label{deg2eta}
\psi_{\bar{\eta}_1} (\varepsilon^1_i) = 0,\;i=0,1,2,\;
\psi_{\bar{\eta}_2} (\varepsilon^2_i) = \begin{cases} \varepsilon^1_0, & {\rm if }\;i=0 \\ 0, & {\rm if }\;i=1,2,3 \end{cases}, \quad
\psi_{\bar{\eta}_3}  (\varepsilon^3_i) = \begin{cases} 0, & {\rm if }\;i=0  \\ \varepsilon^2_{1},& {\rm if }\;i=1 \\ 0,& {\rm if }\;i=2 \\ 0,& {\rm if }\;i=3  \\  \varepsilon^2_{3},& {\rm if }\;i=4 \end{cases} 
\end{equation}
are the first, second and third degrees of the homotopy lifting map $\psi_{\bar{\eta}}$ while
\begin{equation}\label{deg2chi}
\psi_{\bar{\chi}_1} (\varepsilon^1_i) = 0,\;i=0,1,2,\;
\psi_{\bar{\chi}_2} (\varepsilon^2_i) = \begin{cases} 0 & {\rm if }\;i=0 \\ 0, & {\rm if }\;i=1 \\ a\varepsilon^1_1+\varepsilon^1_0 b& {\rm if }\;i=2\\ 0& {\rm if }\;i=3  \end{cases}, \quad
\psi_{\bar{\chi}_3}  (\varepsilon^3_i) = \begin{cases} 0, & {\rm if }\;i=0  \\ 0,& {\rm if }\;i=1 \\ -a \varepsilon^2_{1},& {\rm if }\;i=2 \\  \varepsilon^2_{1}b,& {\rm if }\;i=3  \\  0,& {\rm if }\;i=4 \end{cases}
\end{equation}
are the first, second and third degrees of the homotopy lifting map $\psi_{\bar{\chi}}$. The proof of this claim is by direct computations. We illustrate with just one of the numerous computations i.e we will show that for the cocycle $\bar{\eta}$ and $\bar{\chi}$ with $q=1$, 
\begin{align*}
( d_2\psi_{\bar{\eta}_3} + \psi_{\bar{\eta}_2} d_3)(\varepsilon^3_i) &= (\bar{\eta}\ot 1 - 1\ot\bar{\eta})\Delta_{\K} (\varepsilon^3_i), \quad i=0,1,2,3,4\\
( d_2\psi_{\bar{\chi}_3} + \psi_{\bar{\chi}_2} d_3)(\varepsilon^3_i) &= (\bar{\chi}\ot 1 - 1\ot\bar{\chi})\Delta_{\K} (\varepsilon^3_i), \quad i=0,1,2,3,4.
\end{align*}
\textbf{The case of $\bar{\eta}$:} We get $(d_2\psi_{\bar{\eta}_3} + \psi_{\bar{\eta}_2}d_3)( \varepsilon^{3}_{i})$ equal to
$$\begin{cases} d_2(0)+ \psi_{\bar{\eta}_2}(a\varepsilon^{2}_{0}- \varepsilon^{2}_{0}a) &{\rm if}\;i=0  \\
 d_2(\varepsilon^2_1) +  \psi_{\bar{\eta}_2}(a\varepsilon^{2}_{1}+q\varepsilon^{2}_{1}a +  q^2b\varepsilon^{2}_{0}-\varepsilon^{2}_{0}b) &{\rm if}\;i=1 \\
   d_2(0) +  \psi_{\bar{\eta}_2}(a\varepsilon^{2}_{2}+q^2\varepsilon^{2}_{2}a - q b\varepsilon^{2}_{1}-\varepsilon^{2}_{1}b)&{\rm if}\;i=2 \\  
 d_2(0) + \psi_{\bar{\eta}_2}( b\varepsilon^{2}_{2}-\varepsilon^{2}_{2}b) &{\rm if}\;i=3  \\
  d_2(\varepsilon^2_3) +  \psi_{\bar{\eta}_2}(a\varepsilon^{2}_{3}-\varepsilon^{2}_{0}c) &{\rm if}\;i=4  \end{cases}$$
which is equal to
$$\begin{cases} 0 +  a\varepsilon^{1}_{0} -  \varepsilon^{1}_{0}a\\
 a\varepsilon^1_{1} - q\varepsilon^1_{1}a - qb\varepsilon^1_{0} + \varepsilon^1_{0}b + q^2 b\varepsilon^{1}_{0} - \varepsilon^{1}_{0}b\\ 
0 + 0 \\   
0 + 0 \\ 
a\varepsilon^1_{2} + \varepsilon^1_{0}c   - \varepsilon^{1}_{0}c
\end{cases} 
=   \begin{cases} a\varepsilon^{1}_{0} -\varepsilon^{1}_{0}a &{\rm if}\;i=0 \\ 
 a\varepsilon^{1}_{1} -\varepsilon^{1}_{1}a &{\rm if}\;i=1\\
 0  &{\rm if}\;i=2\\   
0 &{\rm if}\;i=3\\  
a\varepsilon^{1}_{2} &{\rm if}\;i=4 
 \end{cases}.$$
On the other hand, and using Koszul signs in the expansion of  $(1\ot\bar{\eta})(\varepsilon^{n}_{r}\ot\varepsilon^m_s)$  to obtain $ (-1)^{|\bar{\eta}|n}\varepsilon^{n}_{r}\bar{\eta}(\varepsilon^m_s)$, $(\bar{\eta}\ot 1 - 1\ot\bar{\eta})\Delta_{\K} (\varepsilon^{3}_{i})$ is equal to
$$\begin{cases}
 (\bar{\eta}\ot 1 - 1\ot\bar{\eta})  \big[\varepsilon^{0}_{0}\ot\varepsilon^3_0 +  \varepsilon^{1}_{0}\ot\varepsilon^2_0  + \varepsilon^{2}_{0}\ot\varepsilon^1_0 + \varepsilon^{3}_{0}\ot\varepsilon^0_0\big]\\
(\bar{\eta}\ot 1 - 1\ot\bar{\eta})  \big[\varepsilon^{0}_{0}\ot\varepsilon^3_1 +  \varepsilon^{1}_{0}\ot\varepsilon^2_1  +  q^2\varepsilon^{1}_{1}\ot\varepsilon^2_0  + \varepsilon^{2}_{0}\ot\varepsilon^1_1 - q \varepsilon^{2}_{1}\ot\varepsilon^1_0 + \varepsilon^{3}_{1}\ot\varepsilon^0_0\big]  \\
(\bar{\eta}\ot 1 - 1\ot\bar{\eta})  \big[ \varepsilon^{0}_{0}\ot\varepsilon^3_2 +  \varepsilon^{1}_{0}\ot\varepsilon^2_2  + q\varepsilon^{1}_{1}\ot\varepsilon^2_1 - \varepsilon^{2}_{1}\ot\varepsilon^1_1 + \varepsilon^{2}_{2}\ot\varepsilon^1_0  + \varepsilon^{3}_{2}\ot\varepsilon^0_0\big]\\
(\bar{\eta}\ot 1 - 1\ot\bar{\eta})  \big[ \varepsilon^{0}_{0}\ot\varepsilon^3_3 +  \varepsilon^{1}_{1}\ot\varepsilon^2_2  + \varepsilon^{2}_{2}\ot\varepsilon^1_1 + \varepsilon^{3}_{3}\ot\varepsilon^0_0\big]\\
(\bar{\eta}\ot 1 - 1\ot\bar{\eta})  \big[ \varepsilon^{0}_{0}\ot\varepsilon^3_4 +  \varepsilon^{1}_{0}\ot\varepsilon^2_3  + \varepsilon^{2}_{0}\ot\varepsilon^1_2 + \varepsilon^{3}_{4}\ot\varepsilon^0_0\big] \end{cases}$$
which is the same as
$$\begin{cases}  a\varepsilon^1_{0}  - \varepsilon^1_{0}a  &{\rm if}\;i=0  \\ 
 a\varepsilon^1_{1}  - q^2\varepsilon^1_{1}a  &{\rm if}\;i=1\\
 0 &{\rm if}\;i=2 \\ 
 0  &{\rm if}\;i=3 \\ 
 a\varepsilon^1_{2}   &{\rm if}\;i=4. \end{cases}$$
So we see that $( d_2\psi_{\bar{\eta}_3} + \psi_{\bar{\eta}_2} d_3)(\varepsilon^3_i) = (\bar{\eta}\ot 1 - 1\ot\bar{\eta})\Delta_{\K} (\varepsilon^3_i), \quad i=0,1,2,3,4.$\\ 
\textbf{The case of $\bar{\chi}$:}
\begin{align*}
&(d_2\psi_{\chi_3} + \psi_{\bar{\chi}_2}d_3) (\varepsilon^{3}_{i})
= \begin{cases}
d_2(0)+\psi_{\bar{\chi}_2}(a\varepsilon^{2}_{0}- \varepsilon^{2}_{0}a)\\
d_2(0) +\psi_{\bar{\chi}_2}(  a\varepsilon^{2}_{1}+q\varepsilon^{2}_{1}a +  q^2b\varepsilon^{2}_{0}-\varepsilon^{2}_{0}b)\\
d_2(-a\varepsilon^2_{1}) + \psi_{\chi_2}(   a\varepsilon^{2}_{2}+q^2\varepsilon^{2}_{2}a - q b\varepsilon^{2}_{1}-\varepsilon^{2}_{1}b)\\  
d_2(\varepsilon^2_{1}b) +\psi_{\bar{\chi}_2}(  b\varepsilon^{2}_{2}-\varepsilon^{2}_{2}b ) \\
d_2(0) + \psi_{\bar{\chi}_2}(  a\varepsilon^{2}_{3}-\varepsilon^{2}_{0}c ) \end{cases}  \\
&=  \begin{cases}  0 + 0\\
 0+ 0\\
qa\varepsilon^1_{1}a + qab\varepsilon^1_{0} - a\varepsilon^1_{0}b+ a\varepsilon^{1}_{0}b -q^2a\varepsilon^{1}_{1}a - q^2 \varepsilon^{1}_{0}ba\\  
a\varepsilon^1_{1}b - q\varepsilon^1_{1}ab - qb\varepsilon^1_{0}b + ba\varepsilon^{1}_{1} - b\varepsilon^{1}_{0}b - a \varepsilon^{1}_{1}b  \\
0+ 0  \end{cases}
=   \begin{cases}0&{\rm if}\;i=0\\ 0&{\rm if}\;i=1\\
 ab\varepsilon^{1}_{0} -\varepsilon^{1}_{0}ab&{\rm if}\;i=2\\  
 ab\varepsilon^{1}_{1} -\varepsilon^{1}_{1}ab&{\rm if}\;i=3\\ 
0 &{\rm if}\;i=4.  \end{cases}
\end{align*}
On the other hand, and using Koszul signs convention as done in the previous example, 
we get
\begin{align*}
&(\bar{\chi}\ot 1 - 1\ot\bar{\chi})\Delta_{\K} (\varepsilon^{3}_{i}) = (\bar{\chi}\ot 1 - 1\ot\bar{\chi})\\
&\begin{cases}
 \varepsilon^{0}_{0}\ot\varepsilon^3_0 +  \varepsilon^{1}_{0}\ot\varepsilon^2_0  + \varepsilon^{2}_{0}\ot\varepsilon^1_0 + \varepsilon^{3}_{0}\ot\varepsilon^0_0\\
 \varepsilon^{0}_{0}\ot\varepsilon^3_1 +  \varepsilon^{1}_{0}\ot\varepsilon^2_1  +  q^2\varepsilon^{1}_{1}\ot\varepsilon^2_0  + \varepsilon^{2}_{0}\ot\varepsilon^1_1 - q \varepsilon^{2}_{1}\ot\varepsilon^1_0 + \varepsilon^{3}_{1}\ot\varepsilon^0_0  \\
 \varepsilon^{0}_{0}\ot\varepsilon^3_2 +  \varepsilon^{1}_{0}\ot\varepsilon^2_2  + q\varepsilon^{1}_{1}\ot\varepsilon^2_1 - \varepsilon^{2}_{1}\ot\varepsilon^1_1 + \varepsilon^{2}_{2}\ot\varepsilon^1_0  + \varepsilon^{3}_{2}\ot\varepsilon^0_0\\
 \varepsilon^{0}_{0}\ot\varepsilon^3_3 +  \varepsilon^{1}_{1}\ot\varepsilon^2_2  + \varepsilon^{2}_{2}\ot\varepsilon^1_1 + \varepsilon^{3}_{3}\ot\varepsilon^0_0\\
 \varepsilon^{0}_{0}\ot\varepsilon^3_4 +  \varepsilon^{1}_{0}\ot\varepsilon^2_3  + \varepsilon^{2}_{0}\ot\varepsilon^1_2 + \varepsilon^{3}_{4}\ot\varepsilon^0_0 \end{cases}  
 = \begin{cases} 0&{\rm if}\;i=0 \\ 0&{\rm if}\;i=1 \\ ab\varepsilon^1_{0}  - \varepsilon^1_{0}ab &{\rm if}\;i=2\\  ab\varepsilon^1_{1} - \varepsilon^1_{1}ab&{\rm if}\;i=3\\ 0&{\rm if}\;i=4 \end{cases}.
\end{align*}
So we see again that $( d_2\psi_{\bar{\chi}_3} + \psi_{\bar{\chi}_2} d_3)(\varepsilon^3_i) = (\bar{\chi}\ot 1 - 1\ot\bar{\chi})\Delta_{\K} (\varepsilon^3_i), \quad i=0,1,2,3,4.$ 

\subsection{Homotopy liftings from degree 1 cocycles}\label{Gerstenhaber-1-cocycles}
In this section, we will find Hochschild 1 cocycles and their associated homotopy lifting maps. The process of solving equations to obtain cocycles is the same as what was done in the previous case, so we just state the results. Unlike the previous section, we give explicit maps $\psi_{\eta_n}$ and $\psi_{\chi_n}$ for all $n$. 

Suppose that the $\Lambda_q^e$-module homomorphism $\eta:\K_1\rightarrow \Lambda_q$ defined by \\
$ \eta = \begin{pmatrix} \lambda_{0}& \lambda_{1}& \lambda_{2} \end{pmatrix}$  is a cocycle, that is $d^*\eta = 0,$ with $\lambda_i\in\Lambda_q$ for all $i$. Since $d^*\eta : \K_2\rightarrow \Lambda_q$, we solve the equation $d^*\eta(\varepsilon^2_r) = \eta d_2(\varepsilon^2_r) = 0$ and present the solutions in Table \ref{table 2}.
Let us consider the first and second cocycles in Table \ref{table 2}, i.e.
$$ \eta = \begin{pmatrix} a & 0 & 0  \end{pmatrix} \text{  and  }\quad \chi = \begin{pmatrix} ab & 0 & 0 \end{pmatrix}.$$
There are homotopy lifting maps $\psi_\eta, \psi_\chi : \K_{n}\rightarrow \K_{n}$ associated to $\eta$ and $\chi$ respectively satisfying 
\begin{align}
\label{deg1hom1}
(d\psi_\eta - \psi_\eta d )(\varepsilon^n_r) &= (\eta\ot 1 - 1\ot\eta)\Delta_{\K} (\varepsilon^n_r)\qquad \text{and}\\
\label{deg1hom2}
 (d\psi_\chi - \psi_\chi d)(\varepsilon^n_r) &= (\chi\ot 1 - 1\ot\chi)\Delta_{\K} (\varepsilon^n_r).
\end{align}
\begin{table}[t]
\begin{tabular}{ |c|c|c|c|c|c|c|c|c|c| }  \hline
solutions             & 1 & 2  & 3 & 4 & 5 & 6  \\ \hline
$\lambda_0$ & a & ab & 0 & 0 & 0 & 0  \\ 
$\lambda_1$ & 0 & 0   & b &ab& 0 & 0  \\
$\lambda_2$ & 0 & 0   & 0 & 0 & c & bc \\  \hline
\end{tabular}
\caption{Possible values of $\eta(\varepsilon^1_r) = \lambda_r$ for different $r$.}
\label{table 2}
\end{table}
We will prove that for each $n$ and $r$, 
\begin{equation}\label{deg1eta}
\psi_{\eta_n}(\varepsilon^n_r) = \begin{cases} (n-r)\varepsilon^n_r & {\rm when}\;r=0,1,2,\ldots,n \\ (n-1)\varepsilon^n_r & {\rm when}\;r=n+1,\end{cases}
\end{equation}
is a homotopy lifting map associated to the cocycle $\eta$ and whenever $q=1$, 
\begin{equation}\label{deg1chi}
\psi_{\chi_n}(\varepsilon^n_r) = \begin{cases} (\frac{1+(-1)^r}{2})(-1)^{n+1}a\varepsilon^n_{r+1} + (n-r)\varepsilon^n_r b  & \;\;r=0,1,2,\ldots,n-1, \\ 
0 & \;\;r=n,\\ 
(n-1)( b\varepsilon^n_r+\varepsilon^n_1c ) & \;\;r=n+1 ,
\end{cases}
\end{equation}
is a homotopy lifting map associated to the cocycle $\chi.$ We present proofs of these claims as follows.\\
\textbf{The case for $\eta$}. For $r=0,1,\ldots,n$, we get for the left hand side of Equation~\eqref{deg1hom1},
\begin{align*}
& (d\psi_{\eta_n} -  \psi_{\eta_{n-1}} d)(\varepsilon^n_r)  \\
&= d\big\{(n-r) \varepsilon^{n}_r \big\}  - \psi_{\eta_{n-1}} \big\{ \bar{\partial}_{n,r}[a\varepsilon^{n-1}_r+(-1)^{n-r}q^r\varepsilon^{n-1}_r a] +\bar{\partial}_{r,0}[(-q)^{n-r}b\varepsilon^{n-1}_{r-1}+(-1)^{n}\varepsilon^{n-1}_{r-1} b] \big\} \\
& = (n-r) \big\{ \bar{\partial}_{n,r}[a\varepsilon^{n-1}_r +(-1)^{n-r}q^r\varepsilon^{n-1}_r a] +\bar{\partial}_{r,0}[(-q)^{n-r}b\varepsilon^{n-1}_{r-1}+(-1)^{n}\varepsilon^{n-1}_{r-1} b]\big\} \\
& -   \bar{\partial}_{n,r} [(n-r-1) a\varepsilon^{n-1}_r +(-1)^{n-r}q^r(n-r-1)\varepsilon^{n-1}_r a] \\
&+\bar{\partial}_{r,0}[(-q)^{n-r}(n-r) b\varepsilon^{n-1}_{r-1}+(-1)^{n}(n-r)\varepsilon^{n-1}_{r-1} b]  \\
& =   \bar{\partial}_{n,r} \big[ (n-r) - (n-r-1) a\varepsilon^{n-1}_r +(-1)^{n-r}q^r((n-r) - (n-r-1))\varepsilon^{n-1}_r a\big] \\
&+\bar{\partial}_{r,0}\big[(-q)^{n-r}((n-r)-(n-r)) b\varepsilon^{n-1}_{r-1}+(-1)^{n}((n-r)-(n-r))\varepsilon^{n-1}_{r-1} b\big]  \\
& = \bar{\partial}_{n,r} \big[ a\varepsilon^{n-1}_r +(-1)^{n-r}q^r \varepsilon^{n-1}_r a\big] \\
& = a\varepsilon^{n-1}_r +(-1)^{n-r}q^r \varepsilon^{n-1}_r a, \;\text{\;if\;\;} r\neq n.
\end{align*}
If $r=n$, the expression $ (d\psi_{\eta_n} -  \psi_{\eta_{n-1}} d)(\varepsilon^n_r)$ equals $0$. For the case $r=n+1$, we obtain
\begin{align*}
&(d\psi_{\eta_n} -  \psi_{\eta_{n-1}} d)(\varepsilon^n_{n+1})  = d((n-1) \varepsilon^{n}_{n+1})  - \psi_{\eta_{n-1}} ( a\varepsilon^{n-1}_n+(-1)^{n}\varepsilon^{n-1}_0 c) \\
&= (n-1)( a\varepsilon^{n-1}_n+(-1)^{n}\varepsilon^{n-1}_0 c) - (n-2) a\varepsilon^{n-1}_n - (-1)^{n}(n-1)\varepsilon^{n-1}_0 c) \\
&=    a\varepsilon^{n-1}_n.
\end{align*}
On the right hand side of Equation~\eqref{deg1hom1}, we obtain after using the Koszul signs conventions that $(\eta\ot 1 - 1\ot\eta)\Delta_{\K}(\varepsilon^{n}_{r})$ is equal to
\begin{align*}
& (\eta\ot 1 - 1\ot\eta)\Big(\begin{cases}
\displaystyle{\sum_{t=0}^{n} \varepsilon^t_0\ot \varepsilon^{n-t}_0}, &{\rm if}\;r=0\\
\displaystyle{ \sum_{w=0}^{n} \sum_{j=max\{0,s+w-n\}}^{min\{w,s\}} (-q)^{j(n-s+j-w)} \varepsilon^w_j\ot \varepsilon^{n-w}_{s-j}}, &{\rm if}\;0<r<n\\
\displaystyle{\sum_{t=0}^{n} \varepsilon^t_t\ot \varepsilon^{n-t}_{n-t}},&{\rm if}\;r=n\\
 \displaystyle{ \varepsilon^0_0\ot \varepsilon^{n}_{n+1} + \Big[ \sum_{t=0}^{n} \varepsilon^t_0\ot \varepsilon^{n-t}_{n-t+1}\Big] +  \varepsilon^n_{n+1}\ot \varepsilon^{0}_{0}}, &{\rm if}\;r=n+1
\end{cases}\Big).
\end{align*}
In the case $r=0$, substitute $1$ for the index $t$ when applying  $\eta\ot 1$ and substitute $n-1$ for the index $t$ when applying $1\ot\eta$. Similarly for the case $0<r<n$, substitute $1,0,r$ respectively for the indices $w,j,s$ when applying  $\eta\ot 1$ and substitute $n-1,r,r$ respectively for the indices $w,j,s$ when applying $1\ot\eta$. When $r=n$, everything is zero since $\eta(\varepsilon^1_i)=0$ if $i\neq 0$ and finally when $r=n+1$ take substitute $1$ for the index $t$. What we then have is equal to the following
\begin{align*}
&\begin{cases}
 \displaystyle{\eta(\varepsilon^1_0)\varepsilon^{n-1}_0 - (-1)^{n-1}\varepsilon^{n-1}_0\eta(\varepsilon^1_0)}\\
\displaystyle{ \eta(\varepsilon^1_0)\varepsilon^{n-1}_{r} - (-1)^{n-1} (-q)^{r} \varepsilon^{n-1}_r\eta(\varepsilon^{n}_{0})}  \\
\displaystyle{\eta(\varepsilon^1_1)\varepsilon^{n-1}_0 - (-1)^{n-1}\varepsilon^{n-1}_{n-1}\eta(\varepsilon^1_1)} \\\displaystyle{\eta(\varepsilon^1_0)\varepsilon^{n-1}_n } \\
\end{cases}  
=  \begin{cases}
a\varepsilon^{n-1}_0 + (-1)^{n}\varepsilon^{n-1}_0 a &{\rm if}\;r=0\\
a\varepsilon^{n-1}_r +(-1)^{n-r}q^r \varepsilon^{n-1}_r a &{\rm if}\;0<r<n\\
0 &{\rm if}\;r=n\\ 
a\varepsilon^{n-1}_n &{\rm if}\;r=n+1\\
\end{cases}  \\
& =  (d\psi_{\eta_n} -  \psi_{\eta_{n-1}} d) (\varepsilon^n_{r}).
\end{align*}
Thus we have shown that for $r=0,1,\ldots,n+1$, Equation~\eqref{deg1hom1} holds.\\
\\
\textbf{The case for $\chi$.} When $r=0$, the left hand side of Equation~\eqref{deg1hom2} is
\begin{align*}
& (d\psi_{\chi_n} -  \psi_{\chi_{n-1}} d)(\varepsilon^n_0)  \\
&= d((-1)^{n+1}a \varepsilon^{n}_1 + n \varepsilon^{n}_0b  ) -  \psi_\chi( a\varepsilon^{n-1}_0 + (-1)^{n} \varepsilon^{n-1}_0 a)\\
& = (-1)^{n+1} a [a\varepsilon^{n-1}_1 +  (-1)^{n-1}q\varepsilon^{n-1}_1 a + (-q)^{n-1}b\varepsilon^{n-1}_0 + (-1)^{n}\varepsilon^{n-1}_0 b ] \\
&+ n [a\varepsilon^{n-1}_0 + (-1)^n\varepsilon^{n-1}_0 a ]b  - a[(-1)^{n}a\varepsilon^{n-1}_1 + (n-1)\varepsilon^{n-1}_0 b] \\
&- (-1)^{n}[(-1)^{n}a\varepsilon^{n-1}_1 + (n-1)\varepsilon^{n-1}_0 b ] a.
\end{align*}
Whenever $q=1, ab = ba$, so we obtain  $ab \varepsilon^{n-1}_0 + (-1)^{n} \varepsilon^{n-1}_0 ab$
which is equal to the right hand side of~\eqref{deg1hom2}. Therefore $(\chi\ot 1 - 1\ot \chi)\Delta_{\K}(\varepsilon^n_0)$ because
\begin{align*}
&(\chi\ot 1 - 1\ot \chi) \sum_{t=0}^{n} \varepsilon^t_0\ot \varepsilon^{n-t}_0 = \sum_{t=0}^{n} \chi(\varepsilon^{t}_{0}) \varepsilon^{n-t}_{0} -  \sum_{t=0}^{n}(-1)^t \varepsilon^{t}_{0} \chi(\varepsilon^{n-t}_{0}).
\end{align*} 
When $t=1$ in the first sum and $t=n-1$ in the second sum, the last expression is equal to $\displaystyle{ \chi(\varepsilon^{1}_{0}) \varepsilon^{n-1}_{0} - (-1)^{n-1} \varepsilon^{n-1}_{0} \eta(\varepsilon^{1}_{0}) = ab \varepsilon^{n-1}_0 + (-1)^{n} \varepsilon^{n-1}_0 ab}.$\\
When $r$ is even and $0<r<n \;$,  we obtain the following for the left hand side of~\eqref{deg1hom2}
\begin{align*}
& (d\psi_{\chi_n} -  \psi_{\chi_{n-1}} d)(\varepsilon^n_r)  \\
&= d((-1)^{n+1}a \varepsilon^{n}_{r+1} + (n-r) \varepsilon^{n}_r b  ) \\
&-  \psi_\chi( a\varepsilon^{n-1}_r + (-1)^{n-r}q^r \varepsilon^{n-1}_r a + (-q)^{n-r}b\varepsilon^{n-1}_{r-1} + (-1)^n\varepsilon^{n-1}_{r-1} b)\\
& =(-1)^{n+1} a [a\varepsilon^{n-1}_{r+1} +  (-1)^{n-r-1}q^{r+1}\varepsilon^{n-1}_{r+1} a + (-q)^{n-r-1}b\varepsilon^{n-1}_r + (-1)^{n}\varepsilon^{n-1}_r b ] \\
&+ (n-r) [a\varepsilon^{n-r}_r +  (-1)^{n-1}q^r\varepsilon^{n-1}_r a + (-q)^{n-r}b\varepsilon^{n-1}_{r-1} + (-1)^{n}\varepsilon^{n-1}_{r-1} b ]b \\
& - a[(-1)^{n}a\varepsilon^{n-1}_{r+1} + (n-r-1)\varepsilon^{n-1}_r b] - (-1)^{n-r}q^r[(-1)^{n}a\varepsilon^{n-1}_{r+1} + (n-r-1)\varepsilon^{n-1}_r b ] a\\
&- (-q)^{n-r} b [ (n-r)\varepsilon^{n-1}_{r-1} b ] - (-1)^n [ (n-r)\varepsilon^{n-1}_{r-1} b ] b \\
& = [(-1)^{2n-r}q^{r+1} - (-1)^{2n-r}q^r]a\varepsilon^{n-1}_{r+1}a + [(-1)^{n+1}(-q)^{n-r-1} ]ab\varepsilon^{n-1}_{r} \\
&+ [(-1)^{2n+1} +(n-r)-(n-r-1)  ]a\varepsilon^{n-1}_{r}b + [(-1)^{n-r} q^r(n-r)]\varepsilon^{n-1}_{r}ab \\
&+ [-(-1)^{n-r}q^r(n-r-1) ]\varepsilon^{n-1}_{r}ba  + [(-q)^{n-r}(n-r) - (-q)^{n-r}(n-r) ]b\varepsilon^{n-1}_{r-1} b 
\end{align*}
which with $q=1\;(ab=ba)$, we obtain $ab \varepsilon^{n-1}_r + (-1)^{n-r} \varepsilon^{n-1}_r ab$ and all other terms vanish. When $r$ is odd and $o\leq r\leq n$,  we obtain the following for the left hand side of~\eqref{deg1hom2}
\begin{align*}
& (d\psi_{\chi_n} -  \psi_{\chi_{n-1}} d)(\varepsilon^n_r)  \\
&= d( (n-r) \varepsilon^{n}_r b  ) -  \psi_\chi( a\varepsilon^{n-1}_r + (-1)^{n-r}q^r \varepsilon^{n-1}_r a + (-q)^{n-r}b\varepsilon^{n-1}_{r-1} + (-1)^n\varepsilon^{n-1}_{r-1} b)\\
&= (n-r) [a\varepsilon^{n-1}_r +  (-1)^{n-r}q^r\varepsilon^{n-1}_r a + (-q)^{n-r}b\varepsilon^{n-1}_{r-1} + (-1)^{n}\varepsilon^{n-1}_{r-1} b ]b \\
& - a[ (n-r-1)\varepsilon^{n-1}_r b] - (-1)^{n-r}q^r[ (n-r-1)\varepsilon^{n-1}_r b ] a\\
&- (-q)^{n-r} b [(-1)^{n}a\varepsilon^{n-1}_r + (n-r)\varepsilon^{n-1}_{r-1} b ] - (-1)^n [(-1)^{n}a\varepsilon^{n-1}_r + (n-r)\varepsilon^{n-1}_{r-1} b ] b \\
& = [-(-1)^{n}(-q)^{n-r}   ]ba\varepsilon^{n-1}_{r} + [(n-r)-(n-r-1) -(-1)^{2n} ]a\varepsilon^{n-1}_{r}b + [(-1)^{n-r}q^r(n-r)]\varepsilon^{n-1}_{r}ab \\
&+ [-(-1)^{n-r}q^r(n-r-1) ]\varepsilon^{n-1}_{r}ba  + [(-q)^{n-r}(n-r) - (-q)^{n-r}(n-r) ]b\varepsilon^{n-1}_{r-1} b 
\end{align*}
which with $q=1\;$, we get the result obtain previously: $ab \varepsilon^{n-1}_r + (-1)^{n-r} \varepsilon^{n-1}_r ab.$
On the other hand, the right hand side of~\eqref{deg1hom2} when $0<r<n$ becomes
\begin{align*}
& (\chi\ot 1 - 1\ot \chi)\Big[ \sum_{w=0}^{n}\sum_{j=max\{0,r+w-n\}}^{min\{w,r\}} (-q)^{j(n-r+j-w)} \varepsilon^w_j\ot \varepsilon^{n-w}_{r-j}\Big].
\end{align*}
To get a non-zero term, substitute $w=1,j=0$ then apply $\chi\ot1$, and substitute $w=n-1, j=r$ and apply $1\ot\chi$:
\begin{align*}
 (-q)^{0(n-r+1)} \chi(\varepsilon^{1}_{0}) \varepsilon^{n-1}_{r} -  (-1)^{n-1}(-q)^{r(n-r+r-n+1)} \varepsilon^{n-1}_{r} \chi(\varepsilon^{1}_{0}) = ab\varepsilon^{n-1}_r + (-1)^{n-r}\varepsilon^{n-1}_r ab
\end{align*}
When $r=n$, the left hand side of~\eqref{deg1hom2} becomes
\begin{align*}
(d\psi_\chi -  \psi_\chi d)(\varepsilon^n_n)  &= d(0) -  \psi_\chi( b\varepsilon^{n-1}_{n-1} + (-1)^{n} \varepsilon^{n-1}_{n-1} b)  = 0 - b\cdot0 + (-1)^{n+1} 0\cdot b = 0,
\end{align*}
while the right hand side $(\chi\ot 1 - 1\ot \chi)\Delta_{\K}(\varepsilon^n_n)$ becomes
\begin{align*}
&  (\chi\ot 1 - 1\ot \chi) \sum_{t=0}^{n} \varepsilon^t_1\ot \varepsilon^{n-t}_{n-t}= \chi(\varepsilon^{1}_{1}) \varepsilon^{n-1}_{1} - (-1)^{n-1} \varepsilon^{n-1}_{1} \chi(\varepsilon^{1}_{1})  = 0
\end{align*}
and they are equal. It is also true whenever $r=n+1$:
\begin{align*}
&(d\psi_\chi -  \psi_\chi d)(\varepsilon^n_{n+1})  = d((n-1) \varepsilon^{n}_1c + b \varepsilon^{n}_{n+1} ) -  \psi_\chi( a\varepsilon^{n-1}_n + (-1)^{n} \varepsilon^{n-1}_0 c) \\
&= [(n-1)-(n-2)] a \varepsilon^{n-1}_1 c + (n-1)[(-q)^{n-1}+ (-1)^n] b\varepsilon^{n-1}_0 c \\
&+ (n-1)[(-1)^{n} + (-1)^{n-1}]  \varepsilon^{n-1}_0 bc +(n-1) ba\varepsilon^{n-1}_n - (n-2)ab\varepsilon^{n-1}_n \\
& = (n-1-n+2) ab \varepsilon^{n-1}_n  = ab \varepsilon^{n-1}_n
\end{align*}
is equal to 
\begin{align*}
&(\chi\ot 1 - 1\ot \chi)\Delta_{\K}(\varepsilon^n_{n+1})  = (\chi\ot 1 - 1\ot \chi) \sum_{t=0}^{n} \varepsilon^t_0\ot \varepsilon^{n-t}_{n-t+1} +   \varepsilon^n_{n+1}\ot \varepsilon^{0}_{0}\\
&= \chi\ot 1( \varepsilon^1_{0}\ot \varepsilon^{n-1}_{n}) - 1\ot\chi( (-1)^{n-1}\varepsilon^{n-1}_{0}\ot \varepsilon^{1}_{2}) + (\chi\ot 1 - 1\ot \chi)(\varepsilon^{n}_{n+1}\ot \varepsilon^{0}_{0})\\
& = ab \varepsilon^{n-1}_n.
\end{align*}
We have therefore shown that for $r=0,1,\ldots,n+1$, Equation~\eqref{homolift} holds.
\begin{rema} Table \ref{table 3} shows some bracket computations based on these examples. Recall that $\eta = \begin{pmatrix} a & 0 & 0  \end{pmatrix}$ and  $\chi = \begin{pmatrix} ab & 0 & 0 \end{pmatrix}$ are the degree 1 cocycles with homtopy lifting maps given in~\eqref{deg1eta} and~\eqref{deg1chi} respectively. Also take $\overline{\eta} = \begin{pmatrix} a & 0 & 0 & 0 \end{pmatrix}$ and  $\overline{\chi} = \begin{pmatrix} 0& 0 & ab & 0 \end{pmatrix}$ to be the degree 2 cocycles whose homtopy lifting maps were given in~\eqref{deg2eta} and~\eqref{deg2chi} respectively. Take $\theta = \begin{pmatrix} ab & 0 & 0 & 0 \end{pmatrix}$ to be the second degree 2 cocycle appearing in Table \ref{table 1}. The following bracket structure can be verified by direct computations.
\begin{table}
\begin{tabular}{ |c|c|c|c|c|}  \hline
$[\cdot\;,\cdot]$    & $\eta$  &  $\chi$   & $\overline{\eta}$ & $\overline{\chi}$   \\ \hline
$\eta$ & 0 & 0 & $-\overline{\eta}$ & $\overline{\chi}$ \\ 
$\chi$ & 0 & 0 & $-\theta$ & 0 \\  \hline
$\overline{\eta} $ & $\overline{\eta} $ & $\theta$ & 0 & 0 \\  \hline
$\overline{\chi}$ &  -$\overline{\chi}$ & 0 &  0 & 0 \\  \hline
\end{tabular}
\caption{Some bracket computations}
\label{table 3}
\end{table}
For instance $[\bar{\eta},\eta](\varepsilon^2_0) = \bar{\eta}\psi_{\eta}(\varepsilon^2_0) - \eta\psi_{\bar{\eta}}(\varepsilon^2_0) = \bar{\eta}(2\varepsilon^2_0) - \eta(\varepsilon^1_0) = 2a-a = a = \bar{\eta}(\varepsilon^2_0)$ and  $[\bar{\eta},\eta](\varepsilon^2_r)=0$ for all $r\neq 0$.
\end{rema}

\subsection{Connections to derivation operators}\label{deltaoperators}
To compute the Gerstenhaber bracket of a $1$-cocycle and any $n$-cocycle, M. S\'{u}arez-\'{A}lvarez in \cite{MSA} introduced the idea of derivation operators ( called $\delta$-operators in the general settings). We give the definition of a derivation operator in Lemma \ref{Lemmad}. Since an Hochschild $1$-cocycle can be viewed as a derivation, it was shown that there are derivation operators associated to each $1$-cocycle. In particular if $\delta^e:A^e\xrightarrow{}A^e$ is a derivation, that is $\delta^e(ab)=\delta^e(a)b+a\delta^e(b)$ for all $a,b\in A^e$ and $\mathbb{P}\xrightarrow{}A$ is a projective bimodule resolution of $A$, the derivation operators are chain maps lifting $\delta^e$ from $\mathbb{P}_n$ to $\mathbb{P}_n$. Derivation operators are not module homomorphism in general.

In the build up to this work, we in fact realized Equations \eqref{deg1eta} and \eqref{deg1chi} first as derivation operators before observing that they are homotopy lifting maps. It was remarked in \cite[Example 6.3.8]{HCA}, that for an $n$-cocycle $g$, there is an associated homotopy lifting map $\psi_g$ which can also be realized as a coderivation on the tensor algebra of $A$. If $g$ is a $1$-cocycle, then $\psi_g$ may be extended to a derivation operator $\tilde{g}$ in some sense. This describes how a derivation operator may be obtained from a homotopy lifting map. Question \ref{question_homo-doperator} asks whether the converse of this observation is true.

Our goal in this section is to draw connections between derivation operators and homotopy lifting maps  by showing that the homotopy lifting maps of Equations \eqref{deg1eta} and \eqref{deg1chi} associated with $\eta = \begin{pmatrix} a & 0 & 0  \end{pmatrix}$ and $\chi = \begin{pmatrix} ab & 0 & 0 \end{pmatrix}$ respectively are indeed derivation operators.


\begin{lemma}\cite[Lemma 6.2.2]{HCA}\label{Lemmad}
Let $\gamma:\Lambda\xrightarrow{}\Lambda$ be a derivation. There is a $k$-linear chain map $\tilde{\gamma}_{n}:\K_{n}\xrightarrow{}\K_{n}$ lifting $\gamma$ with the property that for each $n$
\begin{equation}\label{d-operators}
\tilde{\gamma}_{n}(a\varepsilon^n_r b) = \gamma(a)\varepsilon^n_r b + a\tilde{\gamma}_n(\varepsilon^n_r) b + a\varepsilon^n_r\gamma(b) 
\end{equation}
for all $a,b\in\Lambda$ and each basis element $\varepsilon^n_r\in \K_n.$ Moreover $\tilde{\gamma}$ is unique up to $\Lambda^e$-module chain homotopy.
\end{lemma}
The chain map $\tilde{\gamma}_{n}$ is called a derivation operator. The following theorem due to M. S\'{u}arez-\'{A}lvarez provides a way for computing the bracket $[\HH^1(\Lambda),\HH^m(\Lambda)]$ on Hochschild cohomology. It was first given in \cite[Lemma 1.8]{MSA} and reformulated by S. Witherspoon in \cite[Theorem 6.2.5]{HCA}.
\begin{theo}\cite{MSA, HCA}\label{derivation_gbrac}
Let $\eta:\Lambda\xrightarrow{}\Lambda$ be a derivation. Let $\K$ be a projective resolution of $\Lambda$ as a $\Lambda^e$-module. Let $\chi\in\HHom_{\Lambda^e}(\K_n,\Lambda)$ be a cocycle and let $\tilde{\eta}_{n}:\K_{n}\xrightarrow{}\K_{n}$ be derivation operators satisfying Equation \eqref{d-operators}. The Gerstenhaber bracket of $\eta$ and $\chi$ is represented by
\begin{equation*}
[\eta,\chi] = \eta\chi - \chi\tilde{\eta}_n
\end{equation*}
as a cocycle on $\K_n$.
\end{theo}
\begin{prop}
When regarded as $k$-linear chain maps, the homotopy lifting maps of Equations \eqref{deg1eta} and \eqref{deg1chi} associated with $\eta = \begin{pmatrix} a & 0 & 0  \end{pmatrix}$ and $\chi = \begin{pmatrix} ab & 0 & 0 \end{pmatrix}$ respectively are indeed derivation operators.
\end{prop}
\begin{proof}\label{homotopylift_doperators}
We make a slight change of notation by defining
\begin{equation}\label{deg1eta-delta}
\tilde{\eta}_n(\varepsilon^n_r) :=\psi_{\eta_n}(\varepsilon^n_r) =  \begin{cases} (n-r)\varepsilon^n_r & {\rm when}\;r=0,1,2,\ldots,n \\ (n-1)\varepsilon^n_r & {\rm when}\;r=n+1.\end{cases}
\end{equation}
to distinguish derivation operators $\tilde{\eta}_{\bullet}$ from homotopy lifting maps $\psi_{\eta_{\bullet}}$. Since $\varepsilon^1_i$ is generated from $(o(f^1_i)\ot_k t(f^1_i))$ which is equal to $a\ot_k a$ when $i=0$, $b\ot_k b$ when $i=1$ and $c\ot_k c$ when $i=2$, we take $\eta = \begin{pmatrix} a & 0 & 0  \end{pmatrix}$ to mean $\eta(a)=a, \eta(b)=0$ and $\eta(c)=0$ as a derivation. The next calculation shows that derivation operators are chain maps i.e. $\tilde{\eta}_{n-1}d_n = d_n\tilde{\eta}_n$ for all $n$. Whenever $0<r\leq n$, the expression $\tilde{\eta}_{n-1}d_n(\varepsilon^n_r)$ is equal to
\begin{align*}
 &\tilde{\eta}_{n-1}\Big[ (1-\partial_{n,r})[a\varepsilon^{n-1}_r+(-1)^{n-r}q^r\varepsilon^{n-1}_r a] + (1-\partial_{r,0})[(-q)^{n-r}b\varepsilon^{n-1}_{r-1}+(-1)^{n}\varepsilon^{n-1}_{r-1} b]\Big]\\
&=  (1-\partial_{n,r})[\eta(a)\varepsilon^{n-1}_r) + a\tilde{\eta}_{n-1}(\varepsilon^{n-1}_r) +(-1)^{n-r}q^r\tilde{\eta}_{n-1}(\varepsilon^{n-1}_r) a + (-1)^{n-r}q^r\varepsilon^{n-1}_r \eta(a)]\\
&+ (1-\partial_{r,0})[(-q)^{n-r}\eta(b)\varepsilon^{n-1}_{r-1} + (-q)^{n-r}b\tilde{\eta}_{n-1}(\varepsilon^{n-1}_{r-1})+(-1)^{n}\tilde{\eta}_{n-1}(\varepsilon^{n-1}_{r-1}) b+(-1)^{n}\varepsilon^{n-1}_{r-1} \eta(b)]\\
&=  (1-\partial_{n,r})[(n-r) a\varepsilon^{n-1}_r  +(-1)^{n-r}q^r(n-r)\varepsilon^{n-1}_ra]\\
&+ (1-\partial_{r,0})[(n-r)(-q)^{n-r}b\varepsilon^{n-1}_{r-1}+(n-r)(-1)^{n}\varepsilon^{n-1}_{r-1} b]\\
&(n-r) \Big[ (1-\partial_{n,r})[a\varepsilon^{n-1}_r+(-1)^{n-r}q^r\varepsilon^{n-1}_r a] + (1-\partial_{r,0})[(-q)^{n-r}b\varepsilon^{n-1}_{r-1}+(-1)^{n}\varepsilon^{n-1}_{r-1} b]\Big]
\end{align*}
while the expression $d_n\tilde{\eta}_n(\varepsilon^n_r) = d_n((n-r)\varepsilon^n_r)$ is equal to
\begin{align*}
 &(n-r) \Big[ (1-\partial_{n,r})[a\varepsilon^{n-1}_r+(-1)^{n-r}q^r\varepsilon^{n-1}_r a] + (1-\partial_{r,0})[(-q)^{n-r}b\varepsilon^{n-1}_{r-1}+(-1)^{n}\varepsilon^{n-1}_{r-1} b]\Big].
\end{align*}
When $r=n+1$, we get $\tilde{\eta}_{n-1}d_n(\varepsilon^n_{n+1})$ equal to
\begin{align*}
 &\tilde{\eta}_{n-1}[a\varepsilon^{n-1}_n+(-1)^{n}\varepsilon^{n-1}_0 c]  =\eta( a)\varepsilon^{n-1}_n+ a\tilde{\eta}_{n-1}(\varepsilon^{n-1}_n)  +(-1)^{n}\tilde{\eta}_{n-1}(\varepsilon^{n-1}_0) c+ (-1)^{n}\varepsilon^{n-1}_0 \eta(c)\\ 
&= (n-1)a\varepsilon^{n-1}_n +(n-1)(-1)^{n}\varepsilon^{n-1}_0 c = (n-1)[a\varepsilon^{n-1}_n+(-1)^{n}\varepsilon^{n-1}_0 c].
\end{align*}
while $d_n\tilde{\eta}_n(\varepsilon^n_{n+1}) = d_n((n-1)\varepsilon^n_{n+1})$ is equal to $(n-1)[a\varepsilon^{n-1}_n+(-1)^{n}\varepsilon^{n-1}_0 c].$

Similarly if we consider $\chi = \begin{pmatrix} ab & 0 & 0  \end{pmatrix}$ to mean $\chi(a)=ab, \chi(b)=0$ and $\chi(c)=0$ as a derivation, the following is a derivation operator associated with $\chi$.
\begin{equation}\label{deg1chi-delta}
\tilde{\chi}_n(\varepsilon^n_r):=\psi_{\chi_n}(\varepsilon^n_r) = \begin{cases} (\frac{1+(-1)^r}{2})(-1)^{n+1}a\varepsilon^n_{r+1} + (n-r)\varepsilon^n_r b  & \;\;r=0,1,2,\ldots,n-1, \\  0 & \;\;r=n,\\  (n-1)( b\varepsilon^n_r+\varepsilon^n_1c ) & \;\;r=n+1 , 
\end{cases}\end{equation}
For instance whenever $0\leq r< n-1$ and $r$ is even, then $\tilde\chi_{n}(\varepsilon^{n}_{r}) = (n-r)\varepsilon^{n}_{r}b + (-1)^{n+1}a\varepsilon^{n}_{r+1}$ and $\tilde\chi_{n}(\varepsilon^{n}_{r-1}) = (n-r+1)\varepsilon^{n}_{r-1}b.$ Therefore the expression $\tilde\chi_{n} d_{n+1}(\varepsilon^{m+1}_{r})$ is equal to
\begin{align*}
 & \tilde\chi_{n} {[}a\varepsilon^{n}_r+(-1)^{n+1-r}\varepsilon^{n}_r a 
 + (-1)^{n+1-r}b\varepsilon^{n}_{r-1}+(-1)^{n+1}\varepsilon^{n}_{r-1} b{]}\\
& = \chi(a)\varepsilon^{n}_r +(-1)^{n+1-r}\varepsilon^{n}_r \chi(a) + a\tilde\chi_n(\varepsilon^{n}_r) +(-1)^{n+1-r}\tilde\chi_n(\varepsilon^{n}_r) a \\
&+\chi(b)\varepsilon^{n}_{r-1} +(-1)^{n+1}\varepsilon^{n}_{r-1} \chi(b) + b\tilde\chi_n(\varepsilon^{n}_{r-1}) +(-1)^{n+1}\tilde\chi_n(\varepsilon^{n}_{r-1}) b\\
& = ab\varepsilon^{n}_r +(-1)^{n+1-r}\varepsilon^{n}_{r} ab + a[(n-r)\varepsilon^{n}_r b + (-1)^{n+1}a\varepsilon^{n}_{r+1}]\\ 
 &+ (-1)^{n-r+1} [(n-r)\varepsilon^{n}_{r}b + (-1)^{n+1}a\varepsilon^{n}_{r+1})]a\\
 & + (-1)^{n-r+1} b[(n-r+1)\varepsilon^{n}_{r-1}b] + (-1)^{n+1}[(n-r+1)\varepsilon^{n}_{r-1}b]b\\
& = (-1)^{2n+2-r} ab\varepsilon^{n}_r +(-1)^{n+1-r}(n+1-r)\varepsilon^{n}_{r} ab + (n-r)a\varepsilon^{n}_r b \\
&+ (-1)^{2n+2-r}a\varepsilon^{n}_{r+1}a  + (-1)^{n-r+1}(n-r+1) b\varepsilon^{n}_{r-1}b,
\end{align*}
which is the same as
\begin{align*}
  d_{n+1}\tilde{\chi}_{n+1}(\varepsilon^{n+1}_{r}) &=  d_{n+1}((n+1-r)\varepsilon^{n+1}_rb +(-1)^{n+2}a\varepsilon^{n+1}_{r+1})\\ 
 &= (n+1-r)(a\varepsilon^{n}_{r}+(-1)^{n+1-r}\varepsilon^{n}_{r} a + (-1)^{n+1-r}b\varepsilon^{n}_{r-1}+(-1)^{n+1}\varepsilon^{n}_{r-1} b)b\\
& + (-1)^{n+2}a (a\varepsilon^{n}_{r+1}+(-1)^{n-r}\varepsilon^{n}_{r+1} a + (-1)^{n-r}b\varepsilon^{n}_{r}+(-1)^{n+1}\varepsilon^{n}_{r} b)\\
& = (-1)^{2n+2-r}ab\varepsilon^{n}_r +(-1)^{n+1-r}(n+1-r)\varepsilon^{n}_{r} ab + (n-r)a\varepsilon^{n}_r b \\
&+ (-1)^{2n+2-r}a\varepsilon^{n}_{r+1}a  + (-1)^{n-r+1}(n-r+1) b\varepsilon^{n}_{r-1}b.
\end{align*}
If it is the case that $r$ is odd, we would have $\tilde\chi_{n}(\varepsilon^{n}_{r}) = (n-r)\varepsilon^{n}_{r}b $ and 
$\tilde\chi_{n}(\varepsilon^{n}_{r-1}) = (n-r+1)\varepsilon^{n}_{r-1}b + (-1)^{n+1}a\varepsilon^{n}_{r}.$ Going through the same calculations, we would have
\begin{align*}
 \tilde\chi_{n} d_{n+1}(\varepsilon^{n+1}_{r}) &= \tilde\chi_{n} {[}a\varepsilon^{n}_r+(-1)^{n+1-r}\varepsilon^{n}_r a 
 + (-1)^{n+1-r}b\varepsilon^{n}_{r-1}+(-1)^{n+1}\varepsilon^{n}_{r-1} b{]}\\
& = \chi(a)\varepsilon^{n}_r +(-1)^{n+1-r}\varepsilon^{n}_r \chi(a) + a\tilde\chi_n(\varepsilon^{n}_r) +(-1)^{n+1-r}\tilde\chi_n(\varepsilon^{n}_r) a \\
&+\chi(b)\varepsilon^{n}_{r-1} +(-1)^{n+1}\varepsilon^{n}_{r-1} \chi(b) + b\tilde\chi_n(\varepsilon^{n}_{r-1}) +(-1)^{n+1}\tilde\chi_n(\varepsilon^{n}_{r-1}) b\\
& = ab\varepsilon^{n}_r +(-1)^{n+1-r}\varepsilon^{n}_{r} ab + a[(n-r)\varepsilon^{n}_r b ] + (-1)^{n-r+1} [(n-r)\varepsilon^{n}_{r}b ]a\\
 & + (-1)^{n-r+1} b[(n-r+1)\varepsilon^{n}_{r-1}b + (-1)^{n+1}a\varepsilon^{n}_{r}] \\
 & + (-1)^{n+1}[(n-r+1)\varepsilon^{n}_{r-1}b + (-1)^{n+1}a\varepsilon^{n}_{r}]b\\
& = (-1)^{n+1-r}(n+1-r)\varepsilon^{n}_{r} ab + (n-r+1)a\varepsilon^{n}_r b  + (-1)^{n-r+1}(n-r+1) b\varepsilon^{n}_{r-1}b\\
& = (n+1-r){[}(-1)^{n+1-r}\varepsilon^{n}_{r} ab + a\varepsilon^{n}_r b  + (-1)^{n-r+1} b\varepsilon^{n}_{r-1}b{]}
\end{align*}
which is the same as
\begin{align*}
 d_{n+1}\tilde\chi_{n+1}(\varepsilon^{n+1}_{r}) &=  d_{n+1}((n-r+1)\varepsilon^{n+1}_{r}b )\\ 
 &= (n-r+1)(a\varepsilon^{n}_{r}+(-1)^{n+1-r}\varepsilon^{n}_{r} a + (-1)^{n+1-r}b\varepsilon^{n}_{r-1}+(-1)^{n+1}\varepsilon^{n}_{r-1} b)b\\
 &= (n-r+1)(a\varepsilon^{n}_{r}b +(-1)^{n+1-r}\varepsilon^{n}_{r} ab + (-1)^{n+1-r}b\varepsilon^{n}_{r-1}b) \\
 & = (n+1-r){[}(-1)^{n+1-r}\varepsilon^{n}_{r} ab + a\varepsilon^{n}_r b  + (-1)^{n-r+1} b\varepsilon^{n}_{r-1}b{]}.
\end{align*}
Whenever $r=n,$ the expression $\tilde\chi_{n} d_{n+1}(\varepsilon^{n+1}_{n+1})$ is equal to
\begin{align*}
& \tilde\chi_{n} {[}b\varepsilon^{n}_{n}+(-1)^{m}\varepsilon^{m}_{m} b{]}= \chi(b)\varepsilon^{n}_{n} + b\tilde\chi_n(\varepsilon^{n}_{n})+(-1)^{n}\tilde\chi_n(\varepsilon^{n}_{n}) b + (-1)^{n}\varepsilon^{n}_{n} \chi(b)  \\
& = 0\cdot\varepsilon^{n}_{n} + b\cdot 0 + (-1)^{n}0\cdot b + (-1)^{n}\varepsilon^{n}_{n} \cdot 0 = 0 =  d_{n+1}\tilde\chi_{n+1}(\varepsilon^{n+1}_{n+1}).
\end{align*}
It is also true that whenever $r=m+1$, the expression $\tilde\chi_{n} d_{n+1}(\varepsilon^{n+1}_{n+2}) $ which is equal to
\begin{align*}
 & \tilde\chi_{n} {[}a\varepsilon^{n}_{n+1}+(-1)^{n+1}\varepsilon^{n}_{0} c{]}= \chi(a)\varepsilon^{n}_{n+1} + (-1)^{n+1}\varepsilon^{n}_{0} \chi(c) + a\tilde\chi_n(\varepsilon^{n}_{n+1})+(-1)^{n+1}\tilde\eta_n(\varepsilon^{n}_{0}) c   \\
& = ab\varepsilon^{n}_{n+1} + a[(n-1)\varepsilon^{n}_{1} c + (n-1)b\varepsilon^{n}_{n+1}] +(-1)^{n+1}[n\varepsilon^{n}_{0} b + (-1)^{n+1}a\varepsilon^{n}_{1} ]c    \\
& = nab\varepsilon^{n}_{n+1} + na\varepsilon^{n}_{1} c + (-1)^{n+1}n\varepsilon^{n}_{0} bc 
\end{align*}
is the same as the expression $ d_{n+1}\tilde\chi_{n+1}(\varepsilon^{n+1}_{n+2}) $ which is given by
\begin{align*}
&d_{n+1}(n\varepsilon^{n+1}_{1}c + nb\varepsilon^{n+1}_{n+2}) = n( a\varepsilon^{n}_{1} + (-1)^{n}\varepsilon^{n}_{1}a + (-1)^n b\varepsilon^{n}_{0} + (-1)^{n+1}\varepsilon^{n}_{0}b] c + n b[a\varepsilon^{n}_{n+1} + (-1)^{n+1}\varepsilon^{n}_{0}c)\\
& = nab\varepsilon^{n}_{n+1} + na\varepsilon^{n}_{1} c + (-1)^{n+1}n\varepsilon^{n}_{0} bc.
\end{align*}
\end{proof}
\begin{rema}
Using the bracket definition by derivation operators given by Theorem \ref{derivation_gbrac}, we realize some computations already given in Table \ref{table 3}. For instance $[\eta,\chi](\varepsilon^1_0) = \eta\chi(\varepsilon^1_0) - \chi\tilde{\eta}_1(\varepsilon^1_0) = \eta(ab) - \chi(\varepsilon^1_0) = \eta(a)b+a\eta(b) - ab = ab + a\cdot 0 - ab = 0$ and $[\eta,\chi](\varepsilon^1_i) = 0$ for $i=1,2.$ So $[\eta,\chi]=0.$ Similarly $[\eta,\bar{\chi}](\varepsilon^2_2) = \eta\bar{\chi}(\varepsilon^2_2) - \bar{\chi}\tilde{\eta}_2(\varepsilon^2_2) = \eta(ab) - \bar{\chi}(0) = \eta(a)b+a\eta(b) - 0 = ab$ and $[\eta,\bar{\chi}](\varepsilon^2_j) = 0$ for $j=0,1,3.$ So $[\eta,\bar{\chi}]=\bar{\chi}.$
\end{rema}

\begin{quest}\label{question_homo-doperator}
Proposition \ref{homotopylift_doperators} shows examples of derivation operators that are homotopy lifting maps. Is this observation true in the following sense: Suppose for a $k$-algebra $A$ and a projective bimodule resolution $\mathbb{P}\xrightarrow{}A$, a derivation or a Hochschild $1$-cocycle $\eta:A\xrightarrow{}A$ has associated to it a chain map $\tilde{\eta}_{n}:\mathbb{P}_{n}\xrightarrow{}\mathbb{P}_{n}$ that is a derivation operator. Since the difference of two derivation operators is a module homomophism, can we find another derivation operator $\tilde{\tau}_n$ associated with $\eta$ such that $\psi_{\eta_n}$ is chain homotopic to $\tilde{\eta}_n - \tilde{\tau}_n$?
\end{quest}\ \\
\\
\textbf{Acknowledgment:} The author expresses profound appreciation to his advisor to Dr. Witherspoon for useful discussions and for reading through the earlier versions of the manuscript.




\end{document}